\documentclass[12pt,a4paper]{amsart}
\usepackage[a4paper,inner=2.8cm,outer=2.5cm,top=2.8cm,bottom=2.5cm]{geometry} 
\usepackage{graphicx,caption,subcaption}              
\usepackage{amsmath,amssymb,amscd,amsthm,amsfonts,anyfontsize,color,dsfont,enumerate,fix-cm,layout,lipsum,lpic,mathrsfs,mdwlist,pigpen,stmaryrd,tensor,tikz,thmtools,xspace}
\usetikzlibrary{arrows,shapes,automata,backgrounds,petri}
\usepackage{hyperref}
\usepackage[all]{xy}
\usepackage{etoolbox}
\apptocmd{\sloppy}{\hbadness 10000\relax}{}{}
\usepackage{mathtools}
\mathtoolsset{showonlyrefs}
\usepackage{xfrac}
\theoremstyle{plain}                 
\newtheorem{theorem}{Theorem}[section]     
\newtheorem{proposition}[theorem]{Proposition} 
\newtheorem{corollary}[theorem]{Corollary}     
\newtheorem{lemma}[theorem]{Lemma}        
\theoremstyle{definition}

\theoremstyle{remark}       
\newtheorem{remark}[theorem]{Remark}    
     
\hypersetup{colorlinks=true,linkcolor=blue,citecolor=purple,filecolor=magenta,urlcolor=cyan}
\renewcommand{\thesubfigure}{\textsc{\roman{subfigure}}}

\newcommand{\mc}[1]{\mathcal{#1}}

\newcommand{\Rel}{\mathord{\mathrm{Rel}}}

\newcommand{\sta}{\mathcal}
\newcommand{\MMM}{\sta M}

\def\Z{\mathbb{Z}}

\def\oM{\overline{\mathcal{M}}}

\def\CP1{\mathbb{C}\mathrm{P}^1}

\def\rsmall{6mm}
\def\rbig{10mm}

\begin{document}
\title[Tautological ring via PPZ relations]{The tautological ring of $\mathcal{M}_{g,n}$ via Pandharipande-Pixton-Zvonkine $r$-spin relations}

\author[R.~Kramer]{Reinier Kramer}

\address[Reinier Kramer]{Korteweg-de Vries Institute for Mathematics, University of Amsterdam, Postbus 94248, 1090 GE Amsterdam, The Netherlands}

\email{R.Kramer@uva.nl}

\author[F.~Labib]{Farrokh Labib}

\address[Farrokh Labib]{Korteweg-de Vries Institute for Mathematics, University of Amsterdam, Postbus 94248, 1090 GE Amsterdam, The Netherlands}

\email{Farrokh.Labib@gmail.com}

\author[D.~Lewanski]{Danilo Lewanski}

\address[Danilo Lewanski]{Korteweg-de Vries Institute for Mathematics, University of Amsterdam, Postbus 94248, 1090 GE Amsterdam, The Netherlands}

\email{D.Lewanski@uva.nl}

\author[S.~Shadrin]{Sergey Shadrin}

\address[Sergey Shadrin]{Korteweg-de Vries Institute for Mathematics, University of Amsterdam, Postbus 94248, 1090 GE Amsterdam, The Netherlands}

\email{S.Shadrin@uva.nl}

\begin{abstract} We use relations in the tautological ring of the moduli spaces $\oM_{g,n}$ derived by Pandharipande, Pixton, and Zvonkine from the Givental formula for the $r$-spin Witten class in order to obtain some restrictions on the dimensions of the tautological rings of the open moduli spaces $\MMM_{g,n}$. In particular, we give a new proof for the result of Looijenga (for $n=1$) and Buryak et al. (for $n\geq 2$) that $\dim R^{g-1}(\MMM_{g,n}) \leq n$. We also give a new proof of the result of Looijenga (for $n=1$) and Ionel (for arbitrary $n\geq 1$) that 
 $R^{i}(\MMM_{g,n}) =0$ for $i\geq g$ and give some estimates for the dimension of $R^{i}(\MMM_{g,n})$ for $i\leq g-2$. 
\end{abstract}

\maketitle

\tableofcontents

\section{Introduction}

The study of the tautological ring $R^*$ of the moduli spaces of curves goes back to the classical papers of Mumford and Faber~\cite{Mumford1983,Faber1999}, see also~\cite{Vakil2003,Pandharipande2002,Zvonkine2012,Tavakol2016}. The tautological ring of the moduli space of curves $\oM_{g,n}$ is additively generated by the so-called dual graphs decorated by $\psi$- and $\kappa$-classes. A dual graph determines a natural stratum in $\oM_{g,n}$, whose vertices correspond to irreducible components of a generic point in the stratum, the leaves correspond to the marked points, and the edges correspond to the nodes. We decorate each vertex with a non-negative integer equal to the geometric genus of the corresponding irreducible component. Each vertex is also equipped with a multi-index $\kappa$-class, and each half-edge, including the leaves, is equipped with a power of the $\psi$-class of the cotangent line bundle at the corresponding marked point or the corresponding branch of the node. There are many linear relations between these generators called tautological relations.\par
We can restrict all these classes to the open moduli space $\MMM_{g,n}$. Then only the graphs with no edges can contribute non-trivially. These graphs just correspond to the classes $\prod_{i=1}^n \psi_i^{d_i} \kappa_{e_1,\dots,e_k}$, $d_i\geq 0$, $e_i\geq 1$. There are still many relations among these classes that can be proved, in particular, that $R^i(\MMM_{g,n})=0$ for $i\geq g$, see~\cite{Looijenga1995,Ionel2002} and also a recent new proof in~\cite{CladerGrushevskyJandaZakharov2016}. In the case $i=g-1$ one can prove that $\dim R^{g-1}(\MMM_{g,n})\leq n$, see~\cite{Looijenga1995,BuryakShadrinZvonkine2016} for the cases $n=1$ and $n\geq 2$ respectively. In this paper we give new proofs of all these results as well as some restrictions on the dimensions of the tautological rings for $i\leq g-2$. Note that, by the non-degeneracy of some matrix of intersection numbers, one can in fact show that $\dim R^{g-1}(\MMM_{g,n})= n$, we refer to~\cite{BuryakShadrinZvonkine2016} for that.\par
We use the tautological relations of Pandharipande-Pixton-Zvonkine~\cite{PPZ16}. Given\-tal-Teleman theory~\cite{Givental2001MMJ,Teleman2012} provides a formula for a homogeneous semi-simple cohomological field theory as a sum over decorated dual graphs as above, see~\cite{Givental2001IMRN,DSS13,DOSS14,PPZ13}. These formulae can be explained as a result of a certain group action on non-homogeneous cohomological field theories applied to the rescaled Gromov-Witten theory of a finite number of points (also known as topological field theory or degree $0$ cohomological field theory), see~\cite{FSZ10,Teleman2012,Shad08,PPZ13}.\par
In some cases we can obtain this way a graphical formula for a cohomological field theory whose properties we know independently. In particular, the graphical formula might contain classes (linear combinations of decorated dual graphs) that are of dimension higher than the homogeneity property allows for a cohomological field theory. Then theses classes must be equal to zero and give us tautological relations. Alternatively, we might consider the graphical formula as a function of some parameter $\phi$ parametrizing a path on the underlying Frobenius manifold with $\phi=0$ lying on the discriminant. If we know independently that the cohomological field theory is defined for any value of $\phi$, including $\phi=0$, then all negative terms of the Laurent series expansion in $\phi$ near $\phi=0$ also give tautological relations. See~\cite{Jand15,Pand16} for some expositions. Once we have a relation for the decorated dual graphs in $\oM_{g,n+m}$, $m\geq 0$, we can multiply it by an arbitrary tautological class, push it forward to $\oM_{g,n}$, and then restrict it to $\MMM_{g,n}$. This gives a relation among the classes $\prod_{i=1}^n \psi_i^{d_i} \kappa_{e_1,\dotsc,e_k}$, $d_i\geq 0$, $e_i\geq 1$, in $R^*(\MMM_{g,n})$. \par
In the case of the Witten $r$-spin class~\cite{Witten1993,PolischukVaintrob2001} the graphical formula and its ingredients are discussed in detail in~\cite{Givental2003,FSZ10,DNOPS15,PPZ16}.\par
Both approaches mentioned above produce the same systems of tautological relations on $\oM_{g,n}$. Two particular paths on the underlying Frobenius manifold are worked out in detail in~\cite{PPZ16}, and we are using one of them in this paper. Note that the results of Janda~\cite{Jand15,Janda2014,Janda2015} guarantee that these relations work in the Chow ring, see a discussion in~\cite{PPZ16}.

\subsection{Organization of the paper}

In section~\ref{section:PPZrelations} we recall the relations of Pand\-ha\-ri\-pan\-de-Pixton-Zvonkine. In section~\ref{sec:dimRg-1}, we use them to give a new proof of the dimension of \( R^{g-1}(\MMM_{g,n} )\), up to one lemma whose proof takes up section~\ref{sec:non-degmatrix}. In section~\ref{sec:dimRg}, we extend this proof scheme to show the vanishing of the tautological ring in all higher degrees. Finally, in section~\ref{sec:bounddimRlow}, we give some bounds for the dimensions of the tautological rings in lower degrees.

\subsection{Acknowledgments} The work of R.~K., D.~L., and S.~S. was supported by the Netherlands Organization for Scientific Research. S.~S. is  grateful to A.~Buryak and D.~Zvonkine for the numerous fruitful discussions on the tautological ring of $\MMM_{g,n}$ and to G.~Carlet, J.~van de Leur, and H.~Posthuma for the numerous fruitful discussions of the Pandaripande-Pixton-Zvonkine relations. 

\section{Pandharipande-Pixton-Zvonkine relations} \label{section:PPZrelations}

In this section we recall the relations in the tautological ring of $\oM_{g,n}$ from~\cite{PPZ16} and put them in a convenient form for our further analysis. 

\subsection{Definition}\label{subsec:DefinitionRelations}

Fix $r\geq 3$. Fix $n$ primary fields $0\leq a_1,\dots,a_n\leq r-2$. All constructions below depend on an auxiliary variable $\phi$ and we fix its exponent $d<0$. A tautological relation $T(g,n,r,a_1,\dots,a_n,d)=0$ depends on these choices, and it is obtained as $T = r^{g-1}\sum_{k=0}^\infty \pi^{(k)}_* T_k/k!$, where $T_k$ is the coefficient of $\phi^d$ in the expression in the decorated dual graphs of $\oM_{g,n+k}$ described below, and $\pi^{(k)}\colon \oM_{g,n+k}\to \oM_{g,n}$ is the natural projection. 

Consider the vector space of primary fields with basis $ \{ e_0,\dots,e_{r-2} \}$. In the basis $\tilde e_i:= \phi^{-i/(r-1)}e_i$ we define the scalar product $\eta_{ij}=\langle \tilde e_i, \tilde e_j \rangle := \phi^{-(r-2)/(r-1)}\delta_{i+j,r-2}$. Equip each vertex of genus $h$ of valency \( v\) in a decorated dual graph with a tensor 
\[
\tilde e_{a_1}\otimes \cdots \otimes \tilde e_{a_v} \mapsto \phi^{(h-1)(r-2)/(r-1)}(r-1)^h \delta_{(r-1) | h-1-\sum_{i=1}^v a_i}.
\]
Define matrices $(R^{-1}_m)^b_a$, $m\geq 0$, $a,b=0,\dots,r-2$, in the basis $\tilde e_0,\dots,\tilde e_{r-2}$. We set $(R^{-1}_m)^b_a=0$ if $b\not\equiv a-m \mod r-1$. If $b\equiv a-m \mod r-1$, then $(R^{-1}_m)^b_a=(r(r-1)\phi^{r/(r-1)})^{-m} P_m(r,a)$, where $P_m(r,a)$, $m\geq 0$, are the polynomials of degree $2m$ in $r,a$ uniquely determined by the following conditions:
\begin{align}
P_0(r,a) & =1;\\
P_m(r,a)-P_{m}(r,a-1) & =((m-\textstyle{\frac 12})r-a)P_{m-1}(r,a-1); \\
P_m(r,0) & =P_m(r,r-2). 
\end{align}
Equip the first $n$ leaves with $\sum_{m=0}^\infty (R^{-1}_m)_{a_i}^b \psi^m e_b$, $i=1,\dots,n$. Equip the $k$ extra leaves (the dilaton leaves) with $-\sum_{m=1}^\infty (R^{-1}_m)_{0}^b \psi^{m+1} e_b$, $i=1,\dots,n$. Equip each edge, where we denote by $\psi'$ and $\psi''$ the $\psi$-classes on the two branches of the corresponding node, with 
\[
\frac{\eta^{i'i''}-\sum_{m',m''=1}^\infty (R^{-1}_{m'})_{j'}^{i'} \eta^{j'j''} (R^{-1}_{m''})_{j''}^{i''} (\psi')^{m'}(\psi'')^{m''} }{\psi'+\psi''} e_{i'}\otimes e_{i''}
\]
Then $T_k$ is defined as the sum over all decorated dual graphs obtained by the contraction of all tensors assigned to their vertices, leaves, and edges, further divided by the order of the automorphism group of the graph.

\subsection{Analysis of relations} \label{subsection:analysisrelations}
There are several observations about the formula introduced in the previous subsection. 
\begin{enumerate}
	\item We obtain a decorated dual graph in $R^D(\oM_{g,n})$ if and only if the sum of the indices of the matrices $R^{-1}_m$ used in its construction is equal to $D$. 
	\item According to~\cite[Theorem 7]{PPZ16}, $T(g,n,r,a_1,\dots,a_n,d)$ is a sum of decorated dual graphs whose coefficients are polynomials in $r,a_1,\dots,a_n$. 
	\item Let $A=\sum_{i=1}^n a_i$. Then $A\equiv g-1+D \mod r-1$. We can assume that $A=g-1+D+x(r-1)$, $x\geq 0$, since $D$ is bounded by $\dim\oM_{g,n}=3g-3+n$, whereas the relations hold for $r$ arbitrarily big.
Collecting the powers of $\phi$ from the contributions above, we obtain $d(r-1)=A+(g-1)(r-2)-rD$. Substituting the expression for $A$, we have that $d < 0$ if and only if $D \geq g + x$. The relevant cases in this paper are the cases $x=0$ and $x=1$.
\end{enumerate}

These relations, valid for particular $r\geq 3$ and $0\leq a_1,\dots,a_n\leq r-2$ are difficult to apply since we have almost no control on the $\kappa$-classes coming from the dilaton leaves. We solve this problem in the following way. 

Let $x=0$, consider the degree $D=g$. We have relations with polynomial coefficients for all $r$ much greater than $g$ and $A=2g-1$.  More precisely, for all integers $0\leq a_1,\dots,a_n\leq 2g-1$, $\sum_{i=1}^n a_i = 2g-1$, we have a relation whose coefficients are polynomials of degree $2g$ in $r$. In other words, we have a polynomial in $r$ whose coefficients are linear combinations of decorated dual graphs in degree $g$, and we can substitute any $r$ sufficiently large. Possible integer values of $r$ determine this polynomial completely, so its evaluation at any other complex value of $r$ is again a relation. 

Let $x=1$, consider the degree $D=g+1$. We have relations with polynomial coefficients for all $r$ much bigger than $g$ and $A=2g-1+r$.  More precisely, for all integers $0\leq a_1,\dots,a_n\leq r-2$, $\sum_{i=1}^n a_i = 2g-1+r$, we have a relation whose coefficients are polynomials of degree $2g+2$ in $r$. 

Note that in both cases we do not, in general, have polynomiality in $a_1,\dots,a_n$, but we have it for some special decorated dual graphs, under some extra conditions. 

We argue below that a good choice of $r$ in both cases is $r=\frac 12$ (note that we still have to explain what we mean in the case $x=1$, since the sum $A$ depends on $r$). In particular, this choice kills all dilaton leaves, and the only non-trivial term that contributes to the sum over $k$ in the definition of $T(g,n,r,a_1,\dots,a_n,d)$ in these cases is $T_0$. 

\subsection{\texorpdfstring{$P$}{R}-polynomials at \texorpdfstring{$r=\frac 12$}{r=1/2}}
%\newpage

%\section{\texorpdfstring{\( P\)}{P}-polynomials at \texorpdfstring{\( r =\frac{1}{2}\)}{r=0.5}}

%We know the shifted Witten spin class given in \cite[proposition 4.6]{PPZ16} as
%\begin{equation}
%\sum_{k \geq 0} \sum_{\substack{\Gamma \in \mathsf{G}_{g,n+k}\\ \text{weightings }\mathbf{a}}} \frac{(r-1)^{1-h^1(\Gamma )}}{| \Aut (\Gamma )|} p_* \bigg\{ \prod_v x_v^{g_v-1} \prod_e \Delta (e) \prod_{i=1}^n L(i) \prod_{i=n+1}^{n+k} K(i) \bigg\}_x,
%\end{equation}
%is concentrated in degrees \( \leq g-1\), by \cite[theorem 9]{PPZ16}. Therefore, the part of this expression of degree \( \geq g\) is a relation in the tautological ring. Furthermore, we know it is polynomial in \( r\) by \cite[lemma 4.7]{PPZ16}. Hence, we can use the value \( r = \frac{1}{2} \) to generate relations.

Recall the \( P_m(r,a)\)-polynomials of \cite{PPZ16} introduced above, and define
\begin{equation}
Q_m(a) \coloneqq \frac{(-1)^m}{2^m m!} \prod_{k=1}^{2m} \Big( a+ 1 - \frac{k}{2} \Big)
\end{equation}

\begin{lemma}
	We have \( P_m\big( \frac{1}{2}, a\big) = Q_m(a)\).
\end{lemma}
\begin{proof}
We will use \cite[lemma 4.3]{PPZ16}. It is clear that \( Q_0(a) = 1\) and are \( Q_m(0) = Q_m\big( -\frac{1}{2}\big) = \delta_{m,0}\). Furthermore
\begin{align}
Q_m(a)- Q_m(a-1) &= \frac{(-1)^m}{2^m m!} \bigg( \prod_{k=1}^{2m} \Big( a+ 1 - \frac{k}{2} \Big) - \prod_{k=1}^{2m} \Big( a - \frac{k}{2} \Big) \bigg)\\
&= \frac{(-1)^m}{2^m m!} \Big( \big( a+\frac{1}{2}\big) a- \big( a-m+\frac{1}{2}\big) (a-m) \Big) \prod_{k=1}^{2m-2} \Big( a - \frac{k}{2} \Big)  \\
&= \frac{1}{2 m} \big(- 2am + m^2 -\frac{1}{2}m\big) Q_{m-1}(a-1)\\
&= \frac{1}{2} \big( m -\frac{1}{2} -2a\big) Q_{m-1}(a-1),
\end{align}
so the equations in the lemma are satisfied.\par
This does not allow us to conclude yet that our \( Q_m(a)\) are equal to the \( P_m(\frac{1}{2}, a)\), as the lemma only states uniqueness for the \( P_m(r,a) \) as polynomials in \( a\) and \( r\). However, we can prove equality by induction on \( m\). The cases for \( m =0,1\) are given explicitly in \cite{PPZ16}, and can be checked easily.\par
Now assume \( m > 1\) and \( P_{m-1} \big(\frac{1}{2}, a\big) = Q_{m-1}(a)\). Then
\begin{equation}
Q_m(a)- Q_m(a-1) = \frac{1}{2} \big( m -\frac{1}{2} -2a\big) Q_{m-1}(a-1),
\end{equation}
with the same relation for \( P_m \big( \frac{1}{2},a\big) \). Hence, \( P_m \big(\frac{1}{2},a\big) = Q_m(a) + c\). Using the same relation for \( m+1\), we get that
\begin{equation}
\Delta_{m+1}(a) \coloneqq P_{m+1}\big( \frac{1}{2},a \big) - Q_{m+1}(a) = -\frac{c}{2}a^2 + \frac{2m-3}{4} ac + d
\end{equation}
We then have that
\begin{align}
0 &= \Delta_{m+1}\big(-\frac{1}{2} \big) - \Delta_{m+1}(0) = -\frac{c}{8} - \frac{2m-3}{8}c = \frac{1-m}{4} c
\end{align}
Because \( m > 1\) by assumption, this proves \( c=0 \), so \( P_m\big(\frac{1}{2},a) = Q_m(a)\).
\end{proof}

\subsection{Simplified relations I}\label{subsec:simprelI}
 In this subsection we discuss the relations that we can obtain from the substitution $r=\frac 12$ for the case of $x=0$ in subsection~\ref{subsection:analysisrelations}. 

The polynomials $Q_m(a)$, $m=0,1,2,\dots$, discussed in the previous subsection, have degree $2m$ and roots $-\frac 12, 0, \frac 12, 1, \dots, m-\frac 32, m-1.$ Note that on the dilaton leaves in the relation of~\cite{PPZ16} we always have a coefficient $(R_m^{-1})_0^i$ for some $m\geq 1$. Since for $r=\frac 12$ we have $(R_m^{-1})_0^i=(-\frac 14 \phi^{-1})^{-m} Q_m(0)=0$, $m\geq 1$, the graphs with dilaton leaves do not contribute to the tautological relations. 

In order to obtain a relation on $\MMM_{g,n}$ we first consider a relation in $\oM_{g,n+m}$ that we push forward to $\oM_{g,n}$ and then restrict to the open moduli space $\MMM_{g,n}$. Note that only graphs that correspond to a partial compactification of $\MMM_{g,n+m}$ can contribute non-trivially. Namely, it is a special case of the rational tails partial compactification, where we require in addition that at most one among the first $n$ marked points can lie on each rational tail. We denote this compactification by $\MMM_{g,n+m}^{\text{rt}[n]}$.

For instance, the dual graphs that can contribute non-trivially to a relation on $\MMM_{g,n+1}^{\text{rt}[n]}$ are either the graph with one vertex and no edges or the graphs with two vertices of genus $g$ and $0$ and one edge connecting them, with leaves labeled by $i$ and $n+1$ attached to the genus $0$ vertex and all other leaves attached to the genus $g$ vertex, $i=1,\dots,n$. These graphs correspond to the divisors in $\MMM_{g,n+1}^{\text{rt}[n]}$ that we denote by $D_{i,n+1}$.

More generally, we denote by $D_I$, $I\subset\{1,\dots,n+m\}$, the divisor in $\oM_{g,n+m}$ whose generic point is represented by a two-component curve, with components of genus $g$ and $0$ connected through a node, such that all the points with labels in $I$ lie on the component of genus $0$, and all other points lie on the component of genus $g$. Then the divisors that belong to $\MMM_{g,n+m}^{\text{rt}[n]}$ are those in which $I$ contains at most one point with a label $1\leq l\leq n$, and all dual graphs that we have to consider are the dual graphs of the generic points of the strata obtained by the intersection of these divisors. 

We denote the relations on \( \oM_{g,n}\) corresponding to the choice of the primary fields $a_1,\dots,a_n$, by $\Omega^{D}_{g,n}(a_1,\dots,a_n)=0$, where $D$ is the degree of the class. In this definition we adjust the coefficient, namely, from now on we ignore the pre-factor $r^{g-1}$ in the definition of the relations, as well as the factor $(-\frac 14 \phi^{-1})^{-D}$ coming from the formula for the $R$-matrices in terms of the polynomials $Q$. Hence, \( \Omega_{g,n}^D(\vec a) \) is proportional to \( T(g,n,\frac12 , \vec a,d(D) )\). We will also often write \( \Omega \) for its restriction to various open parts of the moduli space, such as $\MMM_{g,n+m}^{\text{rt}[n]}$.\par
Note that, as we discussed above, there is a condition on the possible degree of the class and the possible choices of the primary fields implied by the requirement that the degree of the auxiliary parameter $\phi$ must be negative.  

We use the following relations in the rest of the paper:
$\Omega_{g,n+m}^{D}(a_1,\dots,a_{n+m})$, where $D\geq g$, $m\geq 0$, and $\sum_{i=1}^{n+m} a_i = g-1+D$ and all primary fields must be non-negative integers.
We sometimes first multiply these relations by extra monomials of $\psi$-classes before we apply the pushforward to $\oM_{g,n}$ and/or restriction to $\MMM_{g,n}$.

\subsection{Simplified relations II} In this subsection we discuss the relations that we can obtain from the substitution $r=\frac 12$ for the case of $x=1$ in subsection~\ref{subsection:analysisrelations}.

Let us first list all the dual graphs representing the strata in $\MMM_{g,n+2}^{\text{rt}[n]}$, see figure~\ref{fig:chefigura}. Note that under an extra condition on the primary fields $a_1,\dots,a_{n+2}$, namely, that $1\leq a_i\leq r-3-a_{n+1}-a_{n+2}$ for any $1\leq i\leq n$, the coefficients of all these graphs in \( T(g,n+2,r, \vec a,-1 )\), equipped in an arbitrary way with $\psi$- and $\kappa$-classes,  are manifestly polynomial in $a_1,\dots,a_{n+2},r$. Indeed, this extra inequality guarantees that we can uniquely determine the primary fields on the edges in the Givental formula for all these nine graphs.

\begin{figure}[t]
	\begin{subfigure}{0.3\textwidth}
		\renewcommand{\thesubfigure}{\textsc{i}}
		\caption{The entire space}
		\begin{tikzpicture}\label{I}
		\node[draw,circle,minimum size=\rbig] (g) at (0,0) {\( g \)} ;
		\node[] (n+1) at (2,-0.5) {\( a_{n+1}\)};
		\node[] (n+2) at (2,0.5) {\( a_{n+2} \)};
		\node[] (1) at (-0.6,-1.5) {\( a_1 \)};
		\node[] (n) at (0.6,-1.5) {\( a_n\)};
		\node[] (dots) at (0,-1.5) {\( \dotsb \)};
		\path (g) edge (n+1);
		\path (g) edge (n+2);
		\path (g) edge (1);
		\path (g) edge (n);
		\end{tikzpicture}
	\end{subfigure}
	~
	\begin{subfigure}{0.3\textwidth}
		\renewcommand{\thesubfigure}{\textsc{iv}}
		\caption{\( D_{n+1,n+2} \)}
		\begin{tikzpicture}\label{IV}
		\node[draw,circle,minimum size=\rbig] (g) at (0,0) {\( g \)} ;
		\node[draw,circle,minimum size=\rsmall] (b1) at (1.5,0) {\( 0\)};
		\node[] (1) at (-0.6,-1.5) {\( a_1 \)};
		\node[] (n) at (0.6,-1.5) {\( a_n\)};
		\node[] (dots) at (0,-1.5) {\( \dotsb \)};
		\node[] (n+1) at (3,-0.5) {\( a_{n+1}\)};
		\node[] (n+2) at (3,0.5) {\( a_{n+2} \)};
		\path (g) edge (b1);
		\path (g) edge (1);
		\path (g) edge (n);
		\path (b1) edge (n+1);
		\path (b1) edge (n+2);
		\end{tikzpicture}
	\end{subfigure}
	~
	\begin{subfigure}{0.38\textwidth}
		\renewcommand{\thesubfigure}{\textsc{vii}}
		\caption{\( D_{j,n+1,n+2} D_{n+1,n+2} \)}
		\begin{tikzpicture}\label{VII}
		\node[draw,circle,minimum size=\rbig] (g) at (0,0) {\( g \)} ;
		\node[draw,circle,minimum size=\rsmall] (b1) at (1.5,-0.5) {\( 0\)};
		\node[draw,circle,minimum size=\rsmall] (b2) at (3,0) {\( 0\)};
		\node[] (1) at (-0.6,-1.5) {\( a_1 \)};
		\node[] (n) at (0.6,-1.5) {\( a_n\)};
		\node[] (dots) at (0,-1.5) {\( \dotsb \)};
		\node[] (j) at (2.8,-1) {\( a_{j}\)};
		\node[] (n+1) at (4.5,-0.5) {\( a_{n+1}\)};
		\node[] (n+2) at (4.5,0.5) {\( a_{n+2} \)};
		\path (g) edge (b1);
		\path (g) edge (1);
		\path (g) edge (n);
		\path (b2) edge (n+2);
		\path (b1) edge (j);
		\path (b2) edge (n+1);
		\path (b1) edge (b2);
		\end{tikzpicture}
	\end{subfigure}
	\\
	\vspace{12pt}
	\begin{subfigure}{0.3\textwidth}
		\renewcommand{\thesubfigure}{\textsc{ii}}
		\caption{\( D_{j,n+2} \)}
		\begin{tikzpicture}\label{II}
		\node[draw,circle,minimum size=\rbig] (g) at (0,0) {\( g \)} ;
		\node[draw,circle,minimum size=\rsmall] (b1) at (1.5,-0.5) {\( 0\)};
		\node[] (1) at (-0.6,-1.5) {\( a_1 \)};
		\node[] (n) at (0.6,-1.5) {\( a_n\)};
		\node[] (dots) at (0,-1.5) {\( \dotsb \)};
		\node[] (j) at (2.8,-1) {\( a_j\)};
		\node[] (n+2) at (3,0) {\( a_{n+2}\)};
		\node[] (n+1) at (1.5,1) {\( a_{n+1} \)};
		\path (g) edge (b1);
		\path (g) edge (1);
		\path (g) edge (n);
		\path (b1) edge (n+2);
		\path (b1) edge (j);
		\path (g) edge (n+1);
		\end{tikzpicture}
	\end{subfigure}
	~
	\begin{subfigure}{0.3\textwidth}
		\renewcommand{\thesubfigure}{\textsc{v}}
		\caption{\( D_{j,n+1}D_{k,n+2} \)}
		\begin{tikzpicture}\label{V}
		\node[draw,circle,minimum size=\rbig] (g) at (0,0) {\( g \)} ;
		\node[draw,circle,minimum size=\rsmall] (b1) at (1.5,-0.7) {\( 0\)};
		\node[draw,circle,minimum size=\rsmall] (b2) at (1.5,0.7) {\( 0\)};
		\node[] (1) at (-0.6,-1.5) {\( a_1 \)};
		\node[] (n) at (0.6,-1.5) {\( a_n\)};
		\node[] (dots) at (0,-1.5) {\( \dotsb \)};
		\node[] (j) at (2.8,-1.1) {\( a_j\)};
		\node[] (n+1) at (3,-0.4) {\( a_{n+1}\)};
		\node[] (n+2) at (3,1.1) {\( a_{n+2}\)};
		\node[] (k) at (2.8,0.4) {\( a_k \)};
		\path (g) edge (b1);
		\path (g) edge (1);
		\path (g) edge (n);
		\path (b2) edge (n+2);
		\path (b2) edge (k);
		\path (g) edge (b2);
		\path (b1) edge (n+1);
		\path (b1) edge (j);
		\end{tikzpicture}
	\end{subfigure}
	~
	\begin{subfigure}{0.38\textwidth}
		\renewcommand{\thesubfigure}{\textsc{viii}}
		\caption{\( D_{j,n+1,n+2} D_{j,n+2} \)}
		\begin{tikzpicture}\label{VIII}
		\node[draw,circle,minimum size=\rbig] (g) at (0,0) {\( g \)} ;
		\node[draw,circle,minimum size=\rsmall] (b1) at (1.5,0.5) {\( 0\)};
		\node[draw,circle,minimum size=\rsmall] (b2) at (3,0) {\( 0\)};
		\node[] (1) at (-0.6,-1.5) {\( a_1 \)};
		\node[] (n) at (0.6,-1.5) {\( a_n\)};
		\node[] (dots) at (0,-1.5) {\( \dotsb \)};
		\node[] (n+1) at (3,1) {\( a_{n+1}\)};
		\node[] (j) at (4.5,-0.5) {\( a_j\)};
		\node[] (n+2) at (4.5,0.5) {\( a_{n+2} \)};
		\path (g) edge (b1);
		\path (g) edge (1);
		\path (g) edge (n);
		\path (b2) edge (n+2);
		\path (b1) edge (n+1);
		\path (b2) edge (j);
		\path (b1) edge (b2);
		\end{tikzpicture}
	\end{subfigure}
	\\
	\vspace{12pt}
	\begin{subfigure}{0.3\textwidth}
		\renewcommand{\thesubfigure}{\textsc{iii}}
		\caption{\( D_{j,n+1} \)}
		\begin{tikzpicture}\label{III}
		\node[draw,circle,minimum size=\rbig] (g) at (0,0) {\( g \)} ;
		\node[draw,circle,minimum size=\rsmall] (b1) at (1.5,-0.5) {\( 0\)};
		\node[] (1) at (-0.6,-1.5) {\( a_1 \)};
		\node[] (n) at (0.6,-1.5) {\( a_n\)};
		\node[] (dots) at (0,-1.5) {\( \dotsb \)};
		\node[] (j) at (2.8,-1) {\( a_j\)};
		\node[] (n+1) at (3,0) {\( a_{n+1}\)};
		\node[] (n+2) at (1.5,1) {\( a_{n+2} \)};
		\path (g) edge (b1);
		\path (g) edge (1);
		\path (g) edge (n);
		\path (b1) edge (n+1);
		\path (b1) edge (j);
		\path (g) edge (n+2);
		\end{tikzpicture}
	\end{subfigure}
	~
	\begin{subfigure}{0.3\textwidth}
		\renewcommand{\thesubfigure}{\textsc{vi}}
		\caption{\( D_{j,n+1,n+2} \)}
		\begin{tikzpicture}\label{VI}
		\node[draw,circle,minimum size=\rbig] (g) at (0,0) {\( g \)} ;
		\node[draw,circle,minimum size=\rsmall] (b1) at (1.5,0) {\( 0\)};
		\node[] (1) at (-0.6,-1.5) {\( a_1 \)};
		\node[] (n) at (0.6,-1.5) {\( a_n\)};
		\node[] (dots) at (0,-1.5) {\( \dotsb \)};
		\node[] (j) at (2.8,-1) {\( a_j\)};
		\node[] (n+1) at (3,0) {\( a_{n+1} \)};
		\node[] (n+2) at (3,1) {\( a_{n+2}\)};
		\path (g) edge (b1);
		\path (g) edge (1);
		\path (g) edge (n);
		\path (b1) edge (n+1);
		\path (b1) edge (n+2);
		\path (b1) edge (j);
		\end{tikzpicture}
	\end{subfigure}
	~
	\begin{subfigure}{0.38\textwidth}
		\renewcommand{\thesubfigure}{\textsc{ix}}
		\caption{\( D_{j,n+1,n+2} D_{j,n+1} \)}
		\begin{tikzpicture}\label{IX}
		\node[draw,circle,minimum size=\rbig] (g) at (0,0) {\( g \)} ;
		\node[draw,circle,minimum size=\rsmall] (b1) at (1.5,0.5) {\( 0\)};
		\node[draw,circle,minimum size=\rsmall] (b2) at (3,0) {\( 0\)};
		\node[] (1) at (-0.6,-1.5) {\( a_1 \)};
		\node[] (n) at (0.6,-1.5) {\( a_n\)};
		\node[] (dots) at (0,-1.5) {\( \dotsb \)};
		\node[] (n+2) at (3,1) {\( a_{n+2}\)};
		\node[] (j) at (4.5,-0.5) {\( a_{j}\)};
		\node[] (n+1) at (4.5,0.5) {\( a_{n+1} \)};
		\path (g) edge (b1);
		\path (g) edge (1);
		\path (g) edge (n);
		\path (b2) edge (j);
		\path (b1) edge (n+2);
		\path (b2) edge (n+1);
		\path (b1) edge (b2);
		\end{tikzpicture}
	\end{subfigure}
	\caption{Strata in \( \MMM_{g,n+2}^{\text{rt}[2]} \)}
	\label{fig:chefigura}
	\renewcommand{\thesubfigure}{\textsc{\roman{subfigure}}}
\end{figure}

Thus, we have a sequence of tautological relations \( T(g,n+2,r, \vec a,-1 )\) in dimension $g+1$ defined for a big enough $r$, and arbitrary non-negative integers $a_1,\dots,a_{n+2}$ satisfying $a_1+\cdots+a_{n+2}=2g+r-1$ and  $1\leq a_i\leq r-3-a_{n+1}-a_{n+2}$ for any $1\leq i\leq n$.
This gives us enough evaluations of the polynomial coefficients of the decorated dual graphs in $\MMM_{g,n+2}^{\text{rt}[n]}$ to determine these polynomials completely. Thus, we can represent the values of these polynomial coefficients at an arbitrary point $(\tilde a_1,\dots,\tilde a_{n+2},\tilde r)\in \mathbb{C}^{n+3}$ as a linear combination of the Pandharipande-Pixton-Zvonkine relations. This representation is non-unique, since we have too many admissible points $(a_1,\dots,a_{n+2}, r)\in\mathbb{Z}^{n+3}$ satisfying the conditions above. This non-uniqueness is not important for the coefficients of the decorated dual graphs in $\MMM_{g,n+2}^{\text{rt}[n]}$, since we always get the values of their polynomial coefficients at the prescribed points, but the extension of different linear combinations of the relations to the full compactification $\oM_{g,n+2}$ can be different. Indeed, the coefficients of the graphs not listed in figure~\ref{fig:chefigura} can be non-polynomial in $a_1,\dots,a_{n+2}$ (but they are still polynomial in $r$).

We can choose one possible extension to the full compactification $\oM_{g,n+2}$ for each point $(\tilde a_1,\dots,\tilde a_{n+2},\tilde r)\in \mathbb{C}^{n+3}$. In particular, we always specialize $r=\frac 12$, $a_{n+1}=\frac 32$, $a_{n+2}=-\frac 12$. The choice $r=\frac 12$ guarantees that we have no non-trivial dilaton leaves, that is, we have no $\kappa$-classes in the decorations of our graphs. We also divide the whole relation by the factor $(\frac 12) ^{g-1} (-\frac 14 \phi^{-1})^{1-g}$, as in the previous subsection.  

Abusing the notation, we denote these relations by $\Omega^{g+1}_{g,n+2}(a_1,\dots,a_{n},\frac 32, -\frac 12)$. These relations are defined for arbitrary complex numbers $a_1,\dots,a_n$ satisfying $\sum_{i=1}^{n} a_i = 2g-\frac32$. Of course, it is reasonable to use half-integer or integer primary fields $a_1,\dots,a_n$ that would be the roots of the polynomials $Q$, since this gives us a very good control on the possible degrees of the $\psi$-classes on the leaves and the edges of the dual graphs. 

Let us stress once again that restriction of  $\Omega^{g+1}_{g,n+2}(a_1,\dots,a_{n},\frac 32, -\frac 12)$ to $\MMM_{g,n+2}^{\text{rt}[n]}$ is well-defined and can be obtained by the specialization of the polynomial coefficients of the dual graphs in figure~\ref{fig:chefigura} to the point $(a_1,\dots,a_n,a_{n+1}=\frac 32, a_{n+2}=-\frac 12, r=\frac 12)$. We analyze this polynomial coefficients in the next two sections. In the meanwhile, the extension of  $\Omega^{g+1}_{g,n+2}(a_1,\dots,a_{n},\frac 32, -\frac 12)$ from $\MMM_{g,n+2}^{\text{rt}[n]}$ to $\oM_{g,n+2}$ is, in principle, not unique, and we only use that it exists.

\section{The dimension of \texorpdfstring{\( R^{g-1}(\mathcal{M}_{g,n})\)}{g-1-ring}}\label{sec:dimRg-1}

In this section we give a new proof of a result in~\cite{BuryakShadrinZvonkine2016} that \( \dim R^{g-1}(\MMM_{g,n} ) \leq n \).

%Our main tool will be the degree \( g-1+m\) part of the CohFT \( \Omega_{g,n+m}^{\frac{1}{2},\tilde{\tau}}(a_1, \dotsc, a_{n+k})\) on \( \overline{\mathcal{M}}_{g,n+m} \), which we know to be zero, pushed forward to \( \overline{\mathcal{M}}_{g,n} \) and then pulled back to  the open part \( \mathcal{M}_{g,n} \).

\subsection{Reduction to monomials in \texorpdfstring{\( \psi \)}{psi}-classes}\label{subsec:nokappaing-1}
In this subsection we show that any monomial $\psi_1^{d_1}\cdots \psi_n^{d_n}\kappa_{e_1,\dots ,e_m}$ of degree $g-1$ can be expressed as a linear combination of monomials of degree $g-1$ which have only $\psi$-classes. We prove this fact by considering the relations \( \Omega_{g,n+m}^{g-1+m}(a_1, \dotsc, a_{n+m})\) for some appropriate choices of the primary fields. 
%More precisely, we only look at the parts of this CohFT such that after pushforward, forgetting the last $m$ points, we obtain a relation in $R^{g-1}(\mathcal{M}_{g,n})$. This means we only look at the contribution by the graph with one genus $g$ vertex and the contribution from rational tails such that on each tail there is only one of the first $n$ marked points. These rational tails are called relative rational tails.
\begin{proposition}\label{prop:psispanRg-1}
	Let \( g \geq 2\) and \( n \geq 1\). The ring $R^{g-1}(\mathcal{M}_{g,n})$ is spanned by the monomials $\psi_1^{d_1}\cdots \psi_n^{d_n}$ for $d_1,\dots,d_n\geq 0$, $\sum_{i=1}^nd_i=g-1$.
\end{proposition}
\begin{proof}
	The tautological ring of the open moduli space is generated by \( \psi \)- and \( \kappa \)-classes. Hence, a spanning set for the ring \( R^{g-1}(\mathcal{M}_{g,n} )\) is
	\begin{equation}
		\Big\{\psi_1^{d_1}\cdots \psi_n^{d_n}\kappa_{e_1,\dots,e_{m}}\mid m\geq 0, d_i\geq 0, e_j\geq 1, \sum_{i=1}^n d_i+\sum_{j=1}^{m} e_j=g-1\Big\}
	\end{equation}
	Let $V\subset R^{g-1}(\mathcal{M}_{g,n})$ be the subspace spanned by the monomials 
	\begin{equation}
		\Big\{\psi_1^{d_1}\cdots \psi_n^{d_n}\mid \sum_{i=1}^n d_i=g-1\Big\}.
	\end{equation}
	We want to show that $R^{g-1}(\mathcal{M}_{g,n})/V=0$. We do this by induction on the number $m$ of indices of the $\kappa$-class. 

	Let us start with the case \( m = 1\). Consider a relation $\Omega^{g}_{g,n+1}(a_1,\dots,a_{n+1})$ for some admissible choice of the primary fields. In this case we have contributions by the open stratum of smooth curves and by the divisors $D_{n+1,\ell}$, $\ell=1,\dots,n$. The open stratum gives us the following classes:
	\[
	\sum_{\substack{d_1+\dots+d_{n+1}=g\\0\leq d_i\leq a_i}} \prod_{i=1}^{n+1} Q_{d_i}(a_i) \prod_{i=1}^{n+1} \psi_i^{d_i}
	\]
	The condition $d_i\leq a_i$ follows from the fact that $Q_{d}(a)=0$ for $d>a$. The contribution of $D_{n+1,\ell}$ is given by
	\[
	\sum_{\substack{d_1+\dots+d_{n}=g-1\\0\leq d_i \leq a_i+\delta_{i\ell} (a_{n+1}-1) }} \prod_{i=1}^{n} Q_{d_i+\delta_{i\ell}}(a_i+\delta_{i\ell} a_{n+1}) \prod_{i\not=\ell,n+1} \psi_i^{d_i} D_{i,\ell} \pi^*(\psi_\ell^{d_\ell})
	\]
	Here $\pi\colon \MMM^{\textrm{rt}[n]}_{g,n+1}\to \MMM_{g,n}$ is the natural projection. The sum of the pushforwards of these classes to $\MMM_{g,n}$ is equal to 
	\begin{equation}
		0=\sum_{\substack{d_1+\cdots+d_n+e=g-1\\d_i\geq 0, e\geq 1}}\prod_{i=1}^{n} Q_{d_i}(a_i) Q_{e+1}(a_{n+1}) \prod_{i=1}^{n} \psi_i^{d_i}\kappa_e\label{eq:firstgraph}
	\end{equation}
	in $R^{g-1}(\mathcal{M}_{g,n})/V$. Thus we have equation~\eqref{eq:firstgraph} in $R^{g-1}(\mathcal{M}_{g,n})/V$ for each choice of $a_1,\dots,a_{n+1}$ such that $\sum_{i=1}^{n+1} a_i=2g-1$. 
	
	If we choose the lexicographic order on the monomials $\psi_1^{d_1}\cdots\psi_n^{d_n}\kappa_e$, we can then choose the values of the $a_i$ in such a way that the matrix of relations becomes lower triangular, in the following manner. For every monomial \( \psi_1^{d_1} \dotsb \psi_n^{d_n} \kappa_e \), we choose the relation with primary fields \( a_i = d_i \) for \( i =2,\dots,n\), \( a_{n+1} =e+1\), and \( a_1 = d_1+g-1 \). Equation~\eqref{eq:firstgraph} allows to express this monomial in terms of similar monomials with the strictly larger exponent of \( \psi_1 \), so this set of relations does indeed give a lower-triangular matrix. This matrix is invertible, hence all monomials of the form $\psi_1^{d_1}\cdots\psi_n^{d_n}\kappa_e$ are equal to $0$ in $R^{g-1}(\mathcal{M}_{g,n})/V$.
	
	Now assume that all the monomials which have a $\kappa$-class with $m-1$ indices or fewer are equal to $0$ in $R^{g-1}(\mathcal{M}_{g,n})/V$. Consider a relation $\Omega^{g-1+m}_{g,n+1}(a_1,\dots,a_{n},b_1,\dots,b_m)$. This relation, after the push-forward to $\MMM_{g,n}$, gives many terms with no $\kappa$-classes and also with $\kappa$-classes with $\leq m-1$ indices, and also some terms with $\kappa$-classes with $m$ indices. The latter terms are therefore equal to $0$ in $R^{g-1}(\mathcal{M}_{g,n})/V$, namely, we have: 
	\begin{equation}
		0=\sum_{\substack{0\leq d_i\leq a_i\\ 1\leq e_j\leq b_{j}-1}}
		\Big( \prod_{i=1}^n Q_{d_i}(a_i)\psi_{i}^{d_i} \Big) \Big( \prod_{j=1}^m Q_{e_j+1}(b_j) \Big) \kappa_{e_1,\dots,e_m} \label{eq:firstgraph2}
	\end{equation}
for  \( \sum_{i=1}^n d_i + \sum_{j=1}^m e_j = g-1 \). Equation~\eqref{eq:firstgraph2} is valid for each choice of the primary fields $a_i,b_j$ such that  $\sum_{i=1}^n a_i+\sum_{j=1}^m b_j =2g-2+m$.

Choosing a monomial \( \psi_1^{d_1} \dotsb \psi_n^{d_n} \kappa_{e_1, \dotsc, e_m} \), we can choose the primary fields to be \( a_i = d_i \) for \( i =2,\dots,n\), \( b_j = e_j +1\) for \( j=1,\dots,m\), and \( a_1 = d_1 + g -1 \). Again, this relation expresses our monomial as a linear combination of similar monomials with strictly higher exponent of \( \psi_1 \). By downward induction on this exponent, all monomials with \( m\) \( \kappa \)-indices vanish in $R^{g-1}(\mathcal{M}_{g,n})/V$ as well.

Thus $R^{g-1}(\mathcal{M}_{g,n})/V=0$. In other words, any monomial which has a $\kappa$-class as a factor can be expressed as a linear combination of monomials in $\psi$-classes.
\end{proof}

An immediate consequence of this proposition for $n=1$ is the result of Looijenga.

\begin{corollary}[\cite{Looijenga1995}]
 For all \( g \geq 2\), $R^{g-1}(\mathcal{M}_{g,1})=\mathbb{Q} \psi_1^{g-1}$.
\end{corollary}

\subsection{Reduction to \texorpdfstring{$n$}{n} generators}\label{subsec:ngenerators}
In this subsection we prove the following proposition.
\begin{proposition}\label{prop:gen}
For $n \geq 2$ and $g \geq 2$,  every monomial of degree $g-1$ in $\psi$ classes and at most one $\kappa_1$-class can be expressed as linear combinations of the following $n$ classes
\begin{equation}
\psi_1^{g-1}, \,\psi^{g-2}_1 \psi_2,\, \dotsc, \,\psi^{g-2}_1 \psi_n,
\end{equation}
with rational coefficients.
\end{proposition}
Together with the previous subsection this gives a new proof of
\begin{theorem}[\cite{BuryakShadrinZvonkine2016}] For $n \geq 2$ and $g \geq 2$ \label{thm:BSZ16}
\begin{equation}
\dim_{\mathbb{Q}} R^{g-1}(\MMM_{g,n}) \leq n.
\end{equation}
\end{theorem}

\begin{remark}
Note that the possible $\kappa_1$-class is added in proposition~\ref{prop:gen} for a technical reason; it seems to be completely unnecessary in the light of proposition~\ref{prop:psispanRg-1}. In fact, when we include $\kappa_1$, we consider systems of generators approximately twice as large, but this allows us to obtain a much larger system of tautological relations. We do not know of any argument that would allow us to obtain the sufficient number of relations if we consider only monomials of $\psi$-classes as generators. 
\end{remark}

We reduce the number of generators by pushing forward enough relations via the map 
\begin{equation}
\pi^{(2)}_* \colon R^{g+1}(\oM_{g,n+2}) \rightarrow R^{g-1}(\oM_{g,n}),
\end{equation}
where $\pi^{(2)}$ is the forgetful morphism for the last two marked points (we abuse notation a little bit here, restricting the map $\pi^{(2)}$ to $\MMM^{\text{rt}[n]}_{g,n+2} \rightarrow \MMM_{g,n}$).
For $n \geq 2$, let us consider the following vector of primary fields:
\begin{align}
&\vec a \coloneqq \left(a_1 = 2g - \textstyle\frac 32 - A, a_2, \dots, a_n, a_{n+1} = \frac 32, a_{n+2} = -\frac 12 \right)\label{eq:condA},
\end{align}
where $a_i \in \Z_{\geq 0}$, $i = 2, \dots, n$, $A= \sum_{i=2}^n a_i\leq g - 2$.
 We consider the following monomials in $R^{g-1}(\mathcal{M}_{g,n})$:
\begin{align}
y &\coloneqq \psi_1^{g - 2 - A}\prod_{i=2}^n\psi_i^{a_i} \kappa_1,\\
x_\ell &\coloneqq \psi_1^{g-2-A}\prod_{i=2}^n\psi_i^{a_i+\delta_{i\ell}} , \quad  \ell=2, \dots, n. 
\end{align}

\begin{lemma}\label{lem:tautooffboundary}
The tautological relation $\pi^{(2)}_* \Omega_{g, n+2}^{g+1}(\vec a)$, where $\vec a$ is defined in equation~\eqref{eq:condA}, has the following form:
\begin{align} \label{eq:yxell-relation}
& y\cdot \prod_{i=2}^n Q_{a_i}(a_i) Q_2( \textstyle \frac 32) \left( Q_{g - 1 - A}(2g - \frac 32 - A) - Q_{g - 1 - A}(2g - 2 - A)\right) \\
& -\sum_{\ell=2}^n x_\ell \cdot \prod_{\substack{i=2}}^n Q_{a_i+2\delta_{i\ell}}(a_i+ \textstyle\frac 32\delta_{i\ell}) \left( Q_{g - 1 - A}(2g - \frac 32 - A) - Q_{g - 1 - A}(2g - 2 - A)\right) \\
& = \text{ terms divisible by }\psi_1^{g-1-A} .
\end{align}
\begin{proof} In order to prove this lemma we have to analyze all strata in $\MMM_{g,n+2}^{\text{rt}[n]}$. The list of strata is given in figure~\ref{fig:chefigura}. Each stratum should be decorated in all possible ways by the $R$-matrices  with $\psi$-classes as described in section~\ref{section:PPZrelations}.

There are several useful observations that simplify the computation. The leaf labeled by $a_i$, $i=2,\dots,n$, is equipped by $\psi_i^{d_i} Q_{d_i}(a_i)$. This implies that $d_i\leq a_i$. Since $Q_{>2}(\frac 32)=0$ (respectively, $Q_{>0}(-\frac 12)=0$), we conclude that the exponent of $\psi_{n+1}$ is $\leq 2$ (respectively, the exponent of $\psi_{n+2}$ is equal to $0$). Note that we can obtain a monomial with $\kappa_1$-class in the push-forward only if we have $\psi_{n+1}^2$ in the original decorated graph.

Similar observations are also valid for the exponents of the $\psi$-classes at the nodes. Note that there are no $\psi$-classes on the genus $0$ components in any strata except for the case of the dual graph~\subref{VI}, where we must have a $\psi$-class at one of the four points (three marked points and the node) of the genus $0$ component, otherwise the pushforward is equal to $0$. So, for instance, we have $\psi^d$ at the genus $g$ branch of the node on the dual graph~\subref{II} with coefficient $-Q_{d+1}(a_j-\frac 12)$, so in this case $d\leq a_j-1$. If we have $\psi^d$ at the genus $g$ branch of the node on the dual graph~\subref{VIII}, then the product of the coefficients that we have on the edges of this graph is equal to $Q_1(a_j-\frac 12)Q_{d+1}(a_j-\frac 12+\frac 32 -1)$, so in this case $d\leq a_j-1$. And so on; one more example of a detailed analysis of the graphs~\subref{VI}-\subref{IX} is given in lemma~\ref{lem:7terms} in the next section.

%Note that $y$ and $x_l$ are all the monomials of degree $g-1$ with the fixed power $g-2-A$ of $\psi_1$. Let us see that monomials with lower powers of $\psi_1$ contribute trivially. Before pushforward, the class consists of products of \( Q_{d_i}(a_i) \psi_i^{d_i} \) and \( Q_{e+1}(a_i+a_j) (\psi')^e D_{i,j} \). Hence, all \( d_i \leq a_i \). Because of the choice of \( a_{n+1} \) and \( a_{n+2} \), the only way of raising the primary of a point is by forgetting point \( n+1 \) on a divisor \( D_{i,n+1} \), and this would lead to \( Q_{d_i+1}(a_i+\frac{3}{2})\psi_i^{d_i} \), so \( d_i \leq a_i+1 \). Because the total degree is \( g-1\), this and adding \( \kappa_1\) are the only ways to get a monomial not divisible by \( \psi_1^{g-1-A} \), and they correspond to \( x_l \) and \( y\), respectively.\par

We see that we have severe restrictions on the possible powers of $\psi$-classes at all points but the one labeled by $1$, where the exponent is bounded from below, also after the pushforward. Then it is easy to see by the analysis of the graph contributions as above that the exponent of $\psi_1$ is $\geq g-2-A$. Let us list all the terms whose pushforwards to $\MMM_{g,n}$ contain the terms with $\psi_1^{g-2-A}$. 

%Such a monomial would have either a factor \( \kappa_1 \) and at least one exponent of a  \( \psi_i \) larger than \( a_i \), causing a vanishing in of the coefficient \( Q_{> a_i}(a_i) \), or have the two\textemdash{}possibly equal\textemdash{}exponents of \( \psi_i\) raised, which must also imply the occurence of some \( Q_{>a_i}(a_i) = 0\) divisible by $\psi_n^{g-3-A}$, but not by $\psi_n^{g-2-A}$. There are only two cases:
%\begin{enumerate}
%\item[i)] $\psi_1^{a_1} \dots \psi_l^{a_l + 1} \dots \psi_{n-1}^{a_{n-1}} \psi_n^{g - 3 - A} \kappa_1$, for some $l=1, \dots, n-1$.
%\item[ii)] $\psi_1^{a_1} \dots \psi_l^{a_{l_1} + 1} \dots \psi_l^{a_{l_2} + 1}\dots \psi_{n-1}^{a_{n-1}} \psi_n^{g - 3 - A}$ , for some $l_1, l_2=1, \dots, n-1$, possibly equal to each other.
%\end{enumerate}
%In the case of \( y\), $\kappa_1$ should be produced by the pushforward of the $(n+2)$-th leaf with primary field $3/2$, since the $(n+1)$-th leaf can only be decorated by $\psi_{n+1}^0$. Then $\psi_l^{a_l + 1}$ can only be produced as contribution from a stable tree in which the $l$-th leaf and the $(n+1)$-st leaf are isolated on a rational component of the curve, but in that case the edge contribution reads $-Q_{a_l +2}(a_l - 1/2) = 0$. The second case is analogous. 
%Hence, to prove the lemma, it is enough to compute the coefficients of $y$ and $x_l$. We make use of the list of the boundary strata of $\mathcal{M}_{g,n+2}^{\text{rt},[2]}$ pictured in figure \ref{fig:chefigura}.

\begin{itemize}
	\item One of the classes in $\MMM_{g,n+2}^{\text{rt}[n]}$ corresponding to graph~\subref{I} is
	$\psi_1^{g-1-A}\prod_{i=2}^n\psi_i^{a_i}\psi_{n+1}^2$ with  coefficient $\prod_{i=2}^n Q_{a_i}(a_i) Q_2( \textstyle \frac 32) Q_{g - 1 - A}(2g - \frac 32 - A)$. Its pushforward contains the monomial $y$ and the terms divisible by $\psi_1^{g-1-A}$. 
	\item Consider graph~\subref{II} for $j=1$. Let $\pi_{n+2}\colon \oM_{g,n+2}\to\oM_{g,n+1}$ be the forgetful morphism for the $(n+2)$-nd point. One of the classes corresponding to this graph is $\prod_{i=2}^n\psi_i^{a_i}\psi_{n+1}^2D_{1,n+2} (\pi_{n+2})^*(\psi_1^{g-2-A})$ with coefficient $(-1)\prod_{i=2}^n Q_{a_i}(a_i) Q_2( \textstyle \frac 32)\cdot Q_{g - 1 - A}(2g - 2 - A)$. Its pushforward is equal to the monomial $y$.
	\item Let $\pi_{n+1}\colon \oM_{g,n+2}\to\oM_{g,n+1}$ be the forgetful morphism for the $(n+1)$-st point. One of the classes corresponding to graph~\subref{III} for $j=\ell$ is
	$\prod_{i\not=1,\ell}\psi_i^{a_i}\psi_1^{g-1-A} D_{\ell,n+1}\cdot  (\pi_{n+1})^*(\psi_\ell^{a_\ell+1})$ with coefficient 
	$(-1)\prod_{i\not=1,\ell} Q_{a_i}(a_i) Q_{a_\ell+2}(a_\ell+ \textstyle \frac 32) Q_{g - 1 - A}(2g - \frac 32 - A)$. The pushforward of this class contains the monomial $x_\ell$ and the terms divisible by $\psi_1^{g-1-A}$.
	\item Consider graph~\subref{V} for $j=\ell$ and $k=1$. One of the classes corresponding to this graph is $\prod_{i\not=1,\ell}\psi_i^{a_i} D_{\ell,n+1} (\pi_{n+1})^*(\psi_\ell^{a_\ell+1})D_{1,n+2}(\pi_{n+2})^*(\psi_1^{g-2-A})$ with coefficient $\prod_{i\not=1,\ell} Q_{a_i}(a_i) Q_{a_\ell+2}(a_\ell+\frac 32) Q_{g - 1 - A}(2g - 2 - A)$. Its pushforward is equal to the monomial $x_\ell$.
\end{itemize}

Collecting all these terms together, we obtain the left hand side of equation~\eqref{eq:yxell-relation}. Then it is easy to verify case by case that all other graphs and all other possible decorations on these four graphs produce under the push-forward only monomials divisible by $\psi_1^{g-1-A}$.
\end{proof}
\end{lemma}

Let $a_j>0$ for $j=2,\dots,n$. Consider a vector of primary fields $\vec a^{(j)}$ obtained from $\vec a$ by adding $\frac 12$ to $a_1$ and subtracting $\frac 12$ from $a_j$, that is, 
\begin{equation}
\vec{a}^{(j)} := \left(2g-1-A, a_2, \dots, a_{j-1}, a_j - \textstyle\frac12, a_{j+1}, \dots, a_n, \frac 32, -\frac12\right),
\end{equation}

\begin{lemma}\label{lem:j-tautooffboundary}
	The tautological relation $\pi^{(2)}_* \Omega_{g, n+2}^{g+1}(\vec a^{(j)})$ has the following form:
	\begin{align} \label{eq:j-yxell-relation}
	& y\cdot \prod_{i=2}^n Q_{a_i}(a_i-\textstyle \frac 12 \delta_{ij}) Q_2( \textstyle \frac 32) \left( Q_{g - 1 - A}(2g - 1 - A) - Q_{g - 1 - A}(2g - \frac 32 - A)\right) \\
	& -\sum_{\substack{\ell=2\\\ell\not=j}}^n x_\ell \cdot \prod_{\substack{i=2}}^n Q_{a_i+2\delta_{i\ell}}(a_i+ \textstyle\frac 32\delta_{i\ell}-\textstyle \frac 12 \delta_{ij}) \left( Q_{g - 1 - A}(2g - 1 - A) - Q_{g - 1 - A}(2g - \frac 32 - A)\right) \\
	& = \text{ terms divisible by }\psi_1^{g-1-A} .
	\end{align}
\begin{proof}
	The proof of this lemma repeats the proof of lemma~\ref{lem:tautooffboundary}. It is only important to note that the terms that could produce the monomial $x_j$ contribute trivially since they have a factor of $Q_{a_j + 2}(a_j - \frac 12+ \frac 32)=0$ in their coefficients.
\end{proof}
\end{lemma}

%Note that, since $a_j$ is an integer, $Q_{d_j}(a_j)=0$ if and only if $Q_{d_j}(a_j-1/2)=0$. This implies that the stable trees in $\Omega^{g+1}_{g, n+2}(\vec{a}^{(j)})$ and the ones in $\Omega^{g+1}_{g, n+2}(\vec{a})$ are decorated with the same classes, whereas the $j$-th primary field is diminished by $1/2$. Hence, for $a_j \geq 2$, we obtain lemma \ref{lem:tautooffboundary} with the substitution $a_j \mapsto a_j - 1/2$ whenever $a_j$ appears in the argument of $Q$, and no substitution if it appears in the index of $Q$, and similar for the `bulk primary' \( a_1 \). \\
%It is important to note that the term of the monomial $x_j$ contributes trivially, because of the vanishing of $Q_{a_j + 2}(a_j + 3/2 - 1/2)$.

\begin{remark} Note that we have the condition $a_j\geq 0$. Indeed, if $a_j=0$ we can still try to use $\vec a^{(j)}$ as a possible vector of primary fields. But in this case it can contain monomials with lower powers of $\psi_1$, and hence those relations cannot be used for our induction argument in increasing powers of $\psi_1$. To see this,  consider graph~\subref{II}. The coefficient that we have in this case for the degree $d$ of the  $\psi$-class on the genus $g$ branch of the node is equal to $Q_{d+1}(-\frac 12-\frac 12)$. Since $-1$ is not a zero of any polynomial $Q_{\geq 0}$, the degree $d$ can be arbitrarily high, and therefore there is no restriction from below on the degree of $\psi_1$. 
 \end{remark}

%\begin{align}\label{eq:relyxl}
%y &- \sum_{l=1}^{n-1}\frac{Q_{a_l + 2}(a_l + 3/2)}{Q_{a_l}(a_l)Q_2(3/2)} x_l = \text{terms divisible by }\psi_n^{g-1-A}  \\
%y &- \sum_{l=1}^{n-1}\frac{Q_{a_l + 2}(a_l + 3/2)}{Q_{a_l}(a_l)Q_2(3/2)} (1 - \delta_{j,l})x_l = \text{terms divisible by }\psi_n^{g-1-A} , \quad j=1, \dots, n-1. \\
%\end{align}

Let us distinguish now between zero and non-zero primary fields. Up to relabeling the marked points, we can assume that
\begin{equation}
a_2 = a_3 = \dots = a_s = 0,\quad  \text{and } \quad a_i \geq 1, \quad i = s+1, \dots, n.
\end{equation}
Note that, by the definition of the $Q$-polynomials, the coefficient of $y$ is not zero in all relations in lemmata~\ref{lem:tautooffboundary} and~\ref{lem:j-tautooffboundary}.
% the difference $Q_{g - 1 - A}(2g - 3/2 - A) - Q_{g - 1 - A}(2g - 2 - A)$ is equal to
%\begin{equation}
%-\frac 1 2 \frac{(-1)^{g-2-A}}{(g - 2 - A)! 2^{g - 2 - A}}(2g - 3/2 - A)\cdot \dots \cdot (g+1)(g + 1/2),
%\end{equation}
%which is not zero. Similarly, in the case of $\vec{a}^{(j)}$, the difference of $Q$-polynomials with argument shifted by $1/2$ also differs from zero. Note that the factorial is well-defined since $g - 2 - A \geq 0$. Clearly $Q_{a_i}(a_i)$ is also non-zero, for every $i$. 
Dividing these relations by the coefficient of $y$, we obtain the $n-s$ linearly independent relations:
\begin{align}
\Rel_0: \quad y &- \sum_{l=2}^n\frac{Q_{a_l + 2}(a_l + 3/2)}{Q_{a_l}(a_l)Q_2(3/2)} x_l = \text{terms divisible by }\psi_1^{g-1-A}  \\
\Rel_j: \quad y &- \sum_{l=2}^n\frac{Q_{a_l + 2}(a_l + 3/2)}{Q_{a_l}(a_l)Q_2(3/2)} (1 - \delta_{j,l})x_l = \text{terms divisible by }\psi_1^{g-1-A} , 
\end{align}
for $ j=s+1, \dots, n$. Rescaling the generators by rational non-zero coefficients
$$\tilde{x}_l := -\frac{Q_{a_l+2}(a_l + 3/2)}{Q_{a_l}(a_l)Q_2(3/2)}x_l, \quad l=2,\dots,n$$
we can represent the relations in the following matrix:
\begin{center} $M\coloneqq$
	\begin{tabular}{l | c c c c c c c c c}
		&$y$  &$\tilde{x}_2$  &  $\cdots$  & $\tilde{x}_s$ &$\tilde{x}_{s+1}$ &$\tilde{x}_{s+2}$&$\tilde{x}_{s+3}$ &$\cdots$& $\tilde{x}_n$ \\ \hline
		$\Rel_0$& 1 & 1 & $\cdots$  & 1 & 1 & 1&1 &$\cdots$ & 1  \\
		$\Rel_{s+1}$& 1 & 1 & $\cdots$  & 1 & 0 & 1&1 & $\cdots$ & 1  \\
		$\Rel_{s+2}$& 1 & 1 & $\cdots$  & 1 & 1 & 0&1 & $\cdots$ & 1   \\
		$\Rel_{s+3}$& 1 & 1 & $\cdots$  & 1 & 1 & 1&0 & $\cdots$ & 1  \\
		$\vdots$ & $\vdots$ & $\vdots$ & $\ddots$  & $\vdots$ & $\vdots$ & $\vdots$&$\vdots$ & $\ddots$ & $\vdots$ \\
		$\Rel_n$& 1 & 1 & $\cdots$  & 1 & 1 & 1&1 & $\cdots$ & 0 
	\end{tabular}
\end{center}
	Let us take linear combinations of the above relations: $\tilde{\Rel}_j \coloneqq \Rel_0 - \Rel_j$ for $j=s+1, \dots, n$, and $\tilde{\Rel}_0 := \Rel_0 - \sum_{j=s+1}^n \tilde{\Rel}_j$. We obtain:
	\begin{center}
	\begin{tabular}{l | c c c c c c c c c}
		&$y$  &$\tilde{x}_2$  &  $\dots$  & $\tilde{x}_s$ &$\tilde{x}_{s+1}$ &$\tilde{x}_{s+2}$&$\tilde{x}_{s+3}$ &$\dots$& $\tilde{x}_n$ \\  \hline 
		$\tilde{\Rel}_0$& 1 & 1 & $\cdots$  & 1 & 0 & 0&0 & $\cdots$ & 0  \\
		$\tilde{\Rel}_{s+1}$& 0 & 0 & $\cdots$  & 0 & 1 & 0 & 0 & $\cdots$ & 0  \\
		$\tilde{\Rel}_{s+2}$& 0 & 0 &$\cdots$  & 0 & 0 & 1 & 0 & $\cdots$ & 0  \\
		$\tilde{\Rel}_{s+3}$& 0 & 0 &$\cdots$  & 0 & 0 & 0 & 1 & $\cdots$ & 0  \\
		$\vdots$& $\vdots$ & $\vdots$ & $\ddots$  & $\vdots$ & $\vdots$ & $\vdots$&$\vdots$ & $\ddots$ & $\vdots$  \\
		$\tilde{\Rel}_n$& 0 & 0 & $\cdots$  & 0 & 0 & 0 & 0 & $\cdots$ & 1  
	\end{tabular}
	\end{center}
The relation $\tilde{\Rel}_j$ expresses the monomial $\psi_1^{g-2-A} \prod_{i=2}^n \psi_i^{a_i+\delta_{ij}}$ as a linear combination of the generators with higher powers of $\psi_1$. The relation $\tilde{\Rel}_0$ expresses the monomial $\psi_1^{g-2-A}\prod_{i=2}^n \psi_i^{a_i} \kappa_1 $ as linear combination of the monomials $\psi_1^{g-2-A}\prod_{i=2}^n \psi_i^{a_i+\delta_{ij}}$, for $j = 2, \dots, s$ and generators with higher powers of $\psi_1$. In case no primary field $a_i$ is equal to zero (i.~e. $s=1$), any of the monomials $y, x_2, \dots, x_n$ can be expressed in terms of the generators with strictly bigger power of $\psi_1$. 

\subsubsection{Reduction algorithm}

Consider a monomial $\psi_1^{g-1 - \sum d_i} \psi_2^{d_2} \dots \psi_n^{d_n} $. Let $d_M$ be the maximal element in the list of the $d_i$'s with the lowest index. If $d_M \geq 2$, compute the relations $\tilde{\Rel}_j$ for the following vector of primary fields
\begin{align}
&\left(2g - \textstyle\frac 32 - \sum_{i=2}^n d_i, d_2, \dots,d_{M-1}, d_M -1,d_{M+1},\dots, d_n, \,d_{n+1} = \frac 3 2, \, d_{n+2} = -\frac 1 2 \right).
\end{align}
Since $d_M - 1 \geq 1$, we can use the relation $\tilde{\Rel}_{M}$ to express the monomial $ \psi_1^{g-1 - \sum d_i}\cdot \psi_2^{d_2} \cdots \psi_{n}^{d_{n}}$ as a linear combination of monomials with higher powers of $\psi_1$.

We are left to treat the vectors $\vec{d}$ with $d_i = 0$ or $1$, $i=2,\dots,n$. They correspond to the vertices of a unitary $(n-1)$-hypercube with non-negative coordinates. Let $s$ be the number of $d_i$'s equal to zero, so the remaining $(n-1- s)$ $d_i$'s are equal to one, $s=0,\dots,n-1$. Let us distinguish between the different cases in $s$.
\begin{itemize}
\item[$s=n-1$] In this case we have $\psi_1^{g-1}$, a generator. 
\item[$s=n-2$] In this case we have the remaining $n-1$ generators $\psi^{g-2}_1 \psi_i$ for $i=2, \dots, n$. 
\item[$1 \leq s \leq n-3$] This case can be treated as the case $s=0$ for some smaller $n$ discussed below. Let us argue by induction on $n$. For $n\leq 3$, the case $1 \leq s \leq n-3$ does not appear. Let us assume $n \geq 4$. We have at least one zero, so let us assume that $d_j = 0$. Let $\pi^{(1)}_j$ be the morphism that forgets the $j$-th marked point. If the monomial $\psi_1^{g-n+s} \psi_2^{d_2} \dots \hat{\psi_j} \dots \psi_n^{d_n} $ is expressed as linear combination of generators in $R^{g-1}(\MMM_{g, n-1})$ (the space where the point with the label $j$ is forgotten), then the pull-back of this relation via $\pi^{(1)}_j$ expresses $\psi_1^{g-n+s} \psi_2^{d_2} \dots \hat{\psi_j} \dots \psi_n^{d_n}$ as a linear combination of the pull-backs of the the $n-1$ generators of $R^{g-1}(\MMM_{g,n-1})$, $\psi_1^{g-1}$ and $\psi_1^{g-2}\psi_i$, $i\not=1,j$. To conclude we observe that $(\pi^{(1)}_j)^*\psi_1^{g-1}=\psi_1^{g-1}$ and $(\pi^{(1)}_j)^*\psi_1^{g-2}\psi_i=\psi_1^{g-2}\psi_i$, $i\not=1,j$ on the open moduli spaces. Note that
the same reasoning does not work in the case $s = n-2$ since the argument for $s=0$ below uses the assumption $n \geq 3$.
\end{itemize}

\subsubsection{The case \texorpdfstring{$s=0$}{s=0}} For $n \geq 3$, we show that the monomial $\psi_1^{g-n} \prod_{i=2}^n\psi_i^1 $ can be expressed in terms of the generators $\psi_1^{g-1}$, $\psi^{g-2}_1 \psi_i$, $i=2,\dots,n$, concluding this way the proof of proposition~\ref{prop:gen}.  

Let now $\vec{v}_k$ be the vector of primary fields
\begin{align}
\vec{v}_k := \Big( a_1 = 2g - \frac{n+2+k}{2},\underbrace{1, \dots, 1}_{k}, \underbrace{\frac 12, \dots \frac 12}_{n-1-k},  \,a_{n+1} = \frac 3 2, \,a_{n+2} = -\frac 1 2 \Big).
\end{align}
%Denote by $\vec{v}_{1}^{(3)}$ be the vector $\vec{v}_{1}$ in which the second and the third primary fields are exchanged. 
Similarly as before, let
\begin{align}
y &\coloneqq \psi_1^{g - n-1}\prod_{i=2}^n\psi_i^{1} \kappa_1\\
\tilde{x}_\ell &\coloneqq   -\frac{Q_{3}(5/2)}{Q_{1}(1)Q_2(3/2)} \psi_2^{1}  \dotsb  \psi_\ell^{2}  \dotsb  \psi_n^{1}\psi_1^{g-n-1} , \quad  \ell=2, \dots, n. 
\end{align}

Consider the monomials
$$
\psi_1^{g-n}\prod_{i=2}^n\psi_i \quad \text{ and } \quad \psi_1^{g-n} \prod_{i=2}^n \psi_i^{1-\delta_{i\ell}} \kappa_1, \qquad \ell=1,\dots,n-1.
$$
The relations we used in the cases $s\geq 1$ imply that the difference of any two of these  monomials is equal to a linear combination of the generators $\psi_1^{g-1}$, $\psi^{g-2}_1 \psi_i$, $i=2,\dots,n$. 
Let $c_0$ (respectively, $c_1$, $c_2$) be the sum of the coefficients of these monomials 
in the push-forwards of the relations $\Omega_{g,n+2}^{g+1}(\vec{v}_0)$ (respectively, $\Omega_{g,n+2}^{g+1}(\vec{v}_1)$, $\Omega_{g,n+2}^{g+1}(\vec{v}_2)$ ), and let $\hat c_i$ be the normalized coefficients that we get when we divide the relations by the coefficient of $y$. 
 
Now we can expand, in this special case, the system of linear relations collected in the matrix $M$ above. We have a new linear variable, $z\coloneqq \psi_1^{g-n}\prod_{i=2}^n\psi_i = \psi_1^{g-n} \prod_{i=2}^n \psi_i^{1-\delta_{i\ell}} \kappa_1$, $\ell=1,\dots,n-1$, and an extra linear relation $\Rel_*$ corresponding to the vector of primary fields $\vec v_2$. Since in this special case in these relations all the terms with the exponent of $\psi_1$ equal to $g-1-A$, $A=n-1$, are now identified with each other and collected in the variable $z$, these relations express $z,y,x_2,\dots,x_n$ in terms of the monomials proportional to $\psi_1^{g-A}$. The matrix of this system of relations reads:

\begin{center} 
	\begin{tabular}{c | c c c c c c c}
		& $z$ &$y$  & $\tilde{x}_2$  &  $\tilde{x}_3$  &  $\dots$   & $\tilde{x}_n$\\ \hline
		$\Rel_0$& $\hat c_0$ & 1 & 1 & 1  &  \dots & 1  \\
		$\Rel_{2}$ & $\hat c_1$ & 1 & 0 & 1  &  \dots & 1\\
		$\Rel_{3}$& $\hat c_1$ & 1 & 1 & 0  &  \dots  & 1 \\
		\vdots & \vdots & \vdots & \vdots  & \vdots  & $\ddots$  & \vdots \\
		$\Rel_n$& $\hat c_1$ & 1 & 1 & 1  & \dots & 0 \\
		$\Rel_*$ & $\hat c_2$ & 1 & 0& 0 &\dots & 1
	\end{tabular}
\end{center} 
 
This matrix is non-degenerate if and only if $\hat{c}_2 - 2\hat{c}_1 + \hat{c}_0 \neq 0.$ We prove this non-degeneracy in proposition~\ref{prop:nondegeneracy} in the next section. This completes the proof of proposition~\ref{prop:gen} and, as a corollary, theorem~\ref{thm:BSZ16}.

\section{Non-degeneracy of the matrix}\label{sec:non-degmatrix}

In this section we compute the sum of the coefficients of the monomials $\psi_1^{g-n}\prod_{i=2}^n\psi_i$ and $\psi_1^{g-n} \left(\prod_{i=12}^{n}\psi_i^{1-\delta_{i\ell}}\right) \kappa_1$, $\ell=2,\dots,n$ for the three particular sequences of the primary fields. Let us recall the notation.
We denote these sums of coefficients by
\begin{align*}
& c_0 \quad \text{for the primary fields}\quad  g-1 -n,\frac 12,\dots,\frac 12; \\
& c_1 \quad \text{for the primary fields}\quad  g-\frac 32 -n,1, \frac 12,\dots,\frac 12; \\
& c_2 \quad \text{for the primary fields}\quad   g-2 -n,1, 1, \frac 12,\dots,\frac 12.
\end{align*}
We denote the sequence of the primary fields by $a_1,\dots,a_n$. The primary fields at the two points that we forget are as usual $a_{n+1}=\frac 32$ and $a_{n+2}=-\frac 12$. For each $c_i$, $i=0,1,2$, we denote by $\hat c_i$ the normalized coefficient, namely, 
\begin{equation} \label{eq:chatdefinition}
\hat c_i:=c_i \cdot \left((Q_{g+1-n}(a_1)-Q_{g+1-n}(\textstyle{a_1-\frac12}))\prod_{i=2}^nQ_1(a_i)Q_2(\frac 32)\right)^{-1}, \qquad i=0,1,2,
\end{equation}
where the sequence of primary field is exactly the one used for the definition of the corresponding $c_i$, $i=0,1,2$. 

The goal is to prove the following non-degeneracy statement:
\begin{proposition}\label{prop:nondegeneracy} For any $g$ and $n$ satisfying $3\leq n\leq g-1$ we have $\hat c_0-2\hat c_1 + \hat c_2\not=0$.  
\end{proposition}
We prove this proposition below, in subsection~\ref{subsec:proofnondegeneracy}, after we compute the coefficients $c_0$, $c_1$, and $c_2$ explicitly. 

\subsection{A general formula}

First, we prove a general formula for any set of primary fields $a_2,\dots,a_n\in \{\frac 12,1\}$.

\begin{lemma}\label{lem:generalformula} Let all $a_i$, $i=2,\dots, n$ be either $\frac 12$ or $1$. We have $a_1=2g-\frac 32-\sum_{i=2}^n a_i$. A general formula for the sum of the coefficients of the classes $\psi_1^{g-n}\prod_{i=2}^n\psi_i$ and $\psi_1^{g-n}\prod_{i=2}^n\psi_i^{1-\delta_{i\ell}}\kappa_1$, $\ell=1,\dots,n-1$, in the pushforward to $\MMM_{g,n}$ is given by
\begin{align} \label{eq:generalformula}
& \prod_{i=2}^n Q_1(a_i)\cdot \left[ 
(2g-2+n) Q_2\left(\textstyle{\frac 32}\right) Q_{g-n}(a_1) \phantom{\sum_{i=2}^n} \right. \\ \notag
& \phantom{\prod  Q_1(a_i)\cdot [] }+ (2g-2+n) Q_1\left(\textstyle{\frac 32}\right) \left(Q_{g+1-n}(a_1)-Q_{g+1-n}(a_1-\textstyle{\frac 12})\right) \\
& \phantom{\prod  Q_1(a_i)\cdot [] } +Q_{g+2-n}(a_1)-Q_{g+2-n}(a_1-\textstyle{\frac 12}) \\
& \phantom{\prod  Q_1(a_i)\cdot [] } +Q_{g+2-n}(a_1+1)-Q_{g+2-n}(a_1+\textstyle{\frac 32}) \\
& \phantom{\prod  Q_1(a_i)\cdot []} + \left(Q_1\left(\textstyle{\frac 32}\right)-Q_1(1)\right)Q_{g+1-n}(a_1) \\
& + \sum_{\ell=2}^n \frac{ \left( Q_2\left(\frac 32 \right) \left( Q_1(a_\ell) -Q_1(a_\ell-\frac 12) \right) \right)Q_{g-1}(a_1)} {Q_1(a_\ell)} \\
& + \sum_{\ell=2}^n \frac{ \left( Q_3(a_\ell+1) - Q_3(a_\ell+\frac 32) \right)Q_{g-n}(a_1)} {Q_1(a_\ell)} \\
& \left. 
+ \sum_{\ell=2}^n \frac{ \left( Q_2\left(\frac 32 \right)Q_0(a_\ell) - Q_2(a_\ell+\frac 32)\right) \left( Q_{g+1-n}(a_1) -Q_{g+1-n}(a_1-\frac 12) \right)} {Q_1(a_\ell)} 
 \right].
\end{align}	
\end{lemma}

\begin{proof}
The proof of this lemma is based on the analysis of all possible strata in $\overline{\mathcal{M}}_{g,n+2}$ equipped with all possible monomials of $\psi$-classes that can potentially contribute non-trivially to $\psi_1^{g-n}\prod_{i=2}^n\psi_i$ and $\psi_1^{g-n}\prod_{i=2}^n\psi_i^{1-\delta_{i\ell}}\kappa_1$, $\ell=1,\dotsc,n$, under the pushforward. Note that we do not have to consider $\kappa$-classes on the strata in $\overline{\mathcal{M}}_{g,n+2}$ since the choice $r=\frac 12$ guarantees that there are no terms with $\kappa$-classes in the Pandharipande-Pixton-Zvonkine relations. 

Recall that we denote by $D_I$, $I\subset\{1,\dots,n+2\}$, the divisor in $\oM_{g,n+2}$ whose generic point is represented by a two-component curve, with components of genus $g$ and $0$ connected through a node, such that all points with labels in $I$ lie on the component of genus $0$, and all other points lie on the component of genus $g$. In this case we denote by $\psi_0$ the $\psi$-class corresponding to the node on the genus $0$ component.

We denote by $\pi'\colon \oM_{g,n+2} \to \oM_{g,n+1}$ the map forgetting the marked point labeled by $n+2$, by $\pi''\colon \oM_{g,n+1} \to \oM_{g,n}$ the map forgetting the marked point labeled by $n+1$, and by \( \pi \) their composition \( \pi = \pi'' \circ \pi'\). Note that $\pi'_*(\prod_{i=1}^{n+1} \psi_i^{d_i})=\sum_{j:d_j>0} \prod_{i=1}^{n+1} \psi_i^{d_i-\delta_{ij}}$, so, since in order to compute $\pi_*$ we always first apply $\pi'_*$, we typically mention below the degree of which $\psi$-class is reduced. The same we do also for $\pi''_*$ in the relevant cases.

Let us now go through the full list of possible non-trivial contributions. 
%First, we consider only the classes that can potentially give a term with $\kappa_1$. If they, in addition, contain the monomial $\prod_{i=1}^n \psi_i^1 \psi_n^{g-n}$, we mention that as well.
\begin{itemize}
	\item The pushforward of the class $\psi_1^{g-n}\prod_{i=2}^n \psi_i^1  \psi_{n+1}^2$ contains $\psi_1^{g-n}\prod_{i=2}^n \psi_i^1 $ with coefficient $(2g-2+n)\prod_{i=2}^n Q_1(a_i)Q_{g-n}(a_n) Q_2(\frac 32)$. This explains the first line of equation~\eqref{eq:generalformula}. It also contains the terms $\psi_1^{g-n} \prod_{i=2}^{n}\psi_i^{1-\delta_{i\ell}} \kappa_1$, $\ell=2,\dots,n$, with coefficient $\prod_{i=2}^n Q_1(a_i)Q_{g-n}(a_1) Q_2(\frac 32)$.
	\item The pushforward of the class $\psi_1^{g-n}\prod_{i=2}^n \psi_i^{1-\delta_{i\ell}}  D_{\ell,n+2} \psi_{n+1}^2$ also gives the term $\psi_1^{g-n} \prod_{i=2}^{n-1}\psi_i^{1-\delta_{i\ell}}\kappa_1$, with coefficient $-\prod_{i=2}^n Q_1(a_i-\frac{\delta_{i\ell}}2)Q_{g-n}(a_1) Q_2(\frac 32)$. The sum over $\ell$ of this and the previous coefficient is equal to the sixth line of equation~\eqref{eq:generalformula}.
	\item The pushforward of the class $\psi_1^{g+1-n}\prod_{i=2}^{n+1} \psi_i^1 $, where at the first step the map $\pi'_*$ decreases the degree of $\psi_1$, gives $\psi_1^{g-n} \prod_{i=1}^n \psi_i^1 $ with coefficient $(2g-2+n)\prod_{i=2}^n Q_1(a_i)Q_{g+1-n}(a_1) Q_1(\frac 32)$. 
	\item The push-forward of the class $(\pi')^*(\psi_1^{g-n})\prod_{i=2}^{n+1} \psi_i^1 D_{1,n+2}$ gives $\psi_1^{g-n}\prod_{i=2}^n \psi_i^1 $ with coefficient $-(2g-2+n)\prod_{i=2}^n Q_1(a_i)Q_{g+1-n}(a_1-\frac 12) Q_1(\frac 32)$. The sum of this and the previous coefficient is equal to the second line of equation~\eqref{eq:generalformula}.
	\item The pushforward of the class $\psi_1^{g+2-n}\prod_{i=2}^{n} \psi_i^1 $, where both $\pi'_*$ and $\pi''_*$ decrease the degree of $\psi_1$, gives $\psi_1^{g-n} \prod_{i=2}^n \psi_i^1 $ with coefficient $\prod_{i=1}^n Q_1(a_i)Q_{g+2-n}(a_1)$. 
	\item The pushforward of the class $(\pi')^*(\psi_1^{g+1-n})\prod_{i=2}^{n} \psi_i^1 D_{1,n+2} $, where the map $\pi''_*$ decreases the degree of $\psi_1$, gives $\psi_1^{g-n}\prod_{i=2}^n \psi_i^1 $ with coefficient $-\prod_{i=2}^n Q_1(a_i)\cdot Q_{g+2-n}(a_1-\frac 12)$. The sum of this and the previous coefficient is equal to the third line of equation~\eqref{eq:generalformula}.
	\item The pushforward of the class $(\pi'')^*(\psi_1^{g+1-n}) \prod_{i=2}^{n} \psi_i^1 D_{1,n+1} $, where at the first step the map $\pi'_*$ decreases the degree of $(\pi'')^*\psi_1$, gives $\psi_1^{g-n}\prod_{i=2}^n \psi_i^1 $ with coefficient $-\prod_{i=2}^n Q_1(a_i)Q_{g+2-n}(a_1+\frac 32)$. 
	\item Consider the following seven cases together:
	$\pi^*(\psi_1^{g-n}) \prod_{i=2}^{n} \psi_i^1 D_{1,n+1,n+2} \cdot(\psi_0+\psi_1+\psi_{n+1}+\psi_{n+2})$ and
	$\pi^*(\psi_1^{g-n})\prod_{i=2}^{n} \psi_i^1 D_{1,n+1,n+2}(D_{1,n+1}+D_{1,n+2}+D_{n+1,n+2})$. By lemma~\ref{lem:7terms} below, the total sum of their pushforwards is equal to $\psi_1^{g-n} \prod_{i=2}^n \psi_i $ with coefficient $\prod_{i=2}^n Q_1(a_i)Q_{g+2-n}(a_1-\frac 12+\frac 32)$. The sum of this and the previous coefficient is equal to the fourth line in equation~\eqref{eq:generalformula}.
	\item The pushforward of the class $\psi_1^{g-n} \prod_{i=2}^{n} \psi_i^{1-\delta_{i\ell}} D_{\ell,n+1} (\pi'')^*\psi_\ell^1$, where at the first step the map $\pi'_*$ decreases the degree of $(\pi'')^*\psi_\ell$, gives $\psi_1^{g-n}\prod_{i=2}^n \psi_i^1 $ with coefficient $-\prod_{i=2}^n Q_{1+2\delta_{i\ell}}(a_i+\frac{3\delta_{i\ell}}2)Q_{g-n}(a_1)$. 
	\item Consider the following seven cases together:
	$\psi_1^{g-n}\prod_{i=2}^{n} \psi_i^{1-\delta_{i\ell}} D_{\ell,n+1,n+2}\pi^*\psi_\ell^{1} \cdot(\psi_0+\psi_\ell+\psi_{n+1}+\psi_{n+2})$ and
	$\psi_1^{g-n}\prod_{i=2}^{n} \psi_i^{1-\delta_{i\ell}} D_{\ell,n+1,n+2} \cdot (D_{1,n+1}+D_{1,n+2}+D_{n+1,n+2}) \cdot \pi^*\psi_\ell^1$. By lemma~\ref{lem:7terms} below, the total sum of their pushforwards is equal to $ \psi_1^{g-n} \prod_{i=1}^n \psi_i  $ with coefficient $\prod_{i=1}^n Q_{1+2\delta_{i\ell}}(a_i+\delta_{i\ell})Q_{g-n}(a_1)$. Note that this coefficient is always equal to zero, since $Q_3(2)=Q_3(\frac 32)=0$, but we included this term here and in equation~\eqref{eq:generalformula} in any case in order to make the whole formula more transparent and homogeneous. The sum of this and the  previous coefficient is equal to the seventh line in equation~\eqref{eq:generalformula}. 
	\item The pushforward of the class $\psi_1^{g+1-n}\prod_{i=2}^{n-1} \psi_i^1 \psi_{n+1}^1$, where, first, the map $\pi'_*$ decreases the degree of $\psi_{n+1}$, so it becomes zero, and then the map $\pi''_*$ decreases the degree of $\psi_1$, giving $\psi_1^{g-n}\prod_{i=2}^n \psi_i^1 $ with coefficient $\prod_{i=2}^n Q_1(a_i)Q_{g+1-n}(a_1) Q_1(\frac 32)$. 
	\item The pushforward of the class $\psi_1^{g+1-n}\prod_{i=2}^{n-1} \psi_i^1  D_{n+1,n+2}$, where the map $\pi''_*$ decreases the degree of $\psi_{1}$, gives $\psi_1^{g-n} \prod_{i=2}^n \psi_i^1 $ with coefficient 
	$-\prod_{i=2}^n Q_1(a_i)\cdot Q_{g+1-n}(a_1) Q_1(\frac 32-\frac 12).$
	The sum over $\ell$ of this and the previous coefficient is equal to the fifth line of equation~\eqref{eq:generalformula}.
	\item The push-forward of the class  $\psi_1^{g+1-n}\prod_{i=2}^{n-1} \psi_i^{1-\delta_{i\ell}}\psi_{n+1}^2$, where at the first step the map $\pi'_*$ and decreases the degree of $\psi_1$, gives $ \psi_1^{g-n} \prod_{i=2}^{n-1}\psi_i^{1-\delta_{i\ell}}\kappa_1$ with the coefficient $\prod_{i=2}^n Q_{1-\delta_{i\ell}}(a_i)Q_{g+1-n}(a_1) Q_2(\frac 32)$.
	\item The pushforward of the class  $(\pi')^*(\psi_1^{g-n})\prod_{i=2}^{n-1} \psi_i^{1-\delta_{i\ell}}D_{1,n+2}\psi_{n+1}^2$ gives the term  $\psi_1^{g-n} \prod_{i=2}^{n-1}\psi_i^{1-\delta_{i\ell}}\cdot \kappa_1$ with coefficient $-\prod_{i=2}^n Q_{1-\delta_{i\ell}}(a_i)Q_{g+1-n}(a_1-\frac 12) Q_2(\frac 32)$.
	\item The pushforward of the class $\psi_1^{g+1-n}\prod_{i=2}^{n-1} \psi_i^{1-\delta_{i\ell}}D_{\ell,n+1}(\pi'')^*\psi_\ell^1 $, where at the first step the map $\pi'_*$ decreases the degree of $\psi_1$, gives $\psi_1^{g-n}\prod_{i=2}^n \psi_i^1 $ with coefficient $-\prod_{i=2}^n Q_{1-\delta_{i\ell}}(a_i)Q_{g+1-n}(a_1) Q_2(a_\ell+\frac 32)$.
	\item The pushforward of the class $(\pi')^*(\psi_1^{g-n})\prod_{i=2}^{n-1} \psi_i^{1-\delta_{i\ell}}D_{\ell,n+1}(\pi'')^*\psi_\ell^1 D_{1,n+2}$ gives $\psi_1^{g-n} \prod_{i=2}^n \psi_i^1 $ with the coefficient $\prod_{i=2}^n Q_{1-\delta_{i\ell}}(a_i)Q_{g+1-n}(a_1-\frac 12) Q_2(a_\ell+\frac 32)$. The sum over $\ell$ of this and the previous three coefficients is equal to the eighth line in equation~\eqref{eq:generalformula}.
\end{itemize}

Thus we have explained how we obtain all terms in equation~\eqref{eq:generalformula}. Note that since $Q_{\geq 1}(-\frac 12)=0$, we can never have a non-trivial degree of $\psi_{n+2}$ in our formulae. For the same reason, the degree of $\psi_2,\dots,\psi_n$ is bounded from above by $1$ and the degree of $\psi_{n+1}$ is bounded from above by $2$. With this type of reasoning it is easy to see by direct inspection that all other classes of degree $g+1$ do not contain any of the monomials $\psi_1^{g-n}\prod_{i=2}^{n-1}\psi_i$ and $\psi_1^{g-n}\prod_{i=2}^{n-1}\psi_i^{1-\delta_{i\ell}}\kappa_1$, $\ell=2,\dots,n$, with non-trivial coefficients in their push-forwards to $\MMM_{g,n}$. For instance, for an arbitrary $a_\ell$ the class $(\pi'')^*(\psi_1^{g-n})\prod_{i=2}^{n} \psi_i^{1-\delta_{i\ell}}D_{\ell,n+2}(\pi')^*\psi_\ell^1 D_{1,n+1}$ gives as result $\psi_1^{g-n}\prod_{i=2}^n \psi_i^1 $ with the coefficient $\prod_{i=2}^n Q_{1-\delta_{i\ell}}(a_i)Q_{g+1-n}(a_1+\frac 32) Q_2(a_\ell-\frac 12)$. But since $a_\ell$ is either $\frac 12$ or $1$ and $Q_2(0)=Q_2(\frac 12)=0$, this coefficient is equal to zero. 
\end{proof}

\begin{lemma}
	\label{lem:7terms}
	Let the points $1$, $n+1$,  and $n+2$ have arbitrary primary fields $\alpha$, $\beta$, and $\gamma$. Then the pushforward of the part of the class given by
	\begin{align}
	\prod_{i=2}^n \psi_i^{d_i} & \left[  D_{1,n+1,n+2}\pi^*\psi_1^{d_1}(\psi_0+\psi_1+\psi_{n+1}+\psi_{n+2})\right.\\
	& \left. +D_{1,n+1,n+2}(D_{1,n+1} +D_{1,n+2}+D_{n+1,n+2})\pi^*\psi_1^{d_1}\right].
	\end{align}
	is equal to $\prod_{i=1}^{n} \psi_i^{d_i}$ with the coefficient $\prod_{i=2}^n Q_{d_i}(a_i) Q_{d_1+2}(\alpha+\beta+\gamma)$.
\end{lemma}

\begin{proof}
	Indeed, the Givental formula for the deformed $r$-spin class (for a general $r$) in this case implies that these seven summands have the following coefficients, up to a common factor:
	\begin{align}
	\psi_0:\qquad & (R^{-1}_{d_1+2})^{\alpha+\beta+\gamma-d_1-2}_{\alpha+\beta+\gamma}-
	(R^{-1}_{d_1+1})^{\alpha+\beta+\gamma-d_1-2}_{\alpha+\beta+\gamma-1}
	(R^{-1}_{1})^{r-2-(\alpha+\beta+\gamma)}_{r-1-(\alpha+\beta+\gamma)}
	\\
	\psi_1:\qquad & -(R^{-1}_{d_1+1})^{\alpha+\beta+\gamma-d_1-2}_{\alpha+\beta+\gamma-1}
(R^{-1}_{1})^{\alpha-1}_{\alpha}
	\\
\psi_{n+1}:\qquad & -(R^{-1}_{d_1+1})^{\alpha+\beta+\gamma-d_1-2}_{\alpha+\beta+\gamma-1}
(R^{-1}_{1})^{\beta-1}_{\beta}
	\\
\psi_{n+2}:\qquad & -(R^{-1}_{d_1+1})^{\alpha+\beta+\gamma-d_1-2}_{\alpha+\beta+\gamma-1}
(R^{-1}_{1})^{\gamma-1}_{\gamma} \\
D_{1,n+1}: \qquad & (R^{-1}_{d_1+1})^{\alpha+\beta+\gamma-d_1-2}_{\alpha+\beta+\gamma-1}(R^{-1}_1)_{\alpha+\beta}^{\alpha+\beta-1} \\
D_{1,n+2}: \qquad & (R^{-1}_{d_1+1})^{\alpha+\beta+\gamma-d_1-2}_{\alpha+\beta+\gamma-1}(R^{-1}_1)_{\alpha+\gamma}^{\alpha+\gamma-1}\\
D_{n+1,n+2}: \qquad & (R^{-1}_{d_1+1})^{\alpha+\beta+\gamma-d_1-2}_{\alpha+\beta+\gamma-1}(R^{-1}_1)_{\gamma+\beta}^{\gamma+\beta-1}
	\end{align}
	(in addition to a common factor on the right hand side we also omit the common factor $\pi^*(\psi_1^{d_1})\prod_{i=2}^n \psi_i^{d_i} D_{1,n+1,n+2}$ on the left hand side of this table). 
	
	The first term above, $(R^{-1}_{d_n+2})^{\alpha+\beta+\gamma-d_1-2}_{\alpha+\beta+\gamma}$, after the substitution $r=\frac 12$ gives us the factor $Q_{d_1+2}(\alpha+\beta+\gamma)$, and times the common factor of $\prod_{i=2}^n Q_{d_i}(a_i)$ it is exactly the results we state in the lemma.  We have to show that the other seven terms sum up to zero. Indeed, the other seven terms, after substitution $r=\frac 12$, are proportional to 
	\begin{align}
	& Q_1(\textstyle{-\frac 12}-\alpha-\beta-\gamma)+Q_1(\alpha)+Q_1(\beta)+Q_1(\gamma) \\
	& -Q_1(\alpha+\beta)-Q_1(\alpha+\gamma)-Q_1(\gamma+\beta)
	\end{align}
	Note that $Q_1(-\frac 12-x)=Q_1(x)$, so the expression above is proportional to 
	\begin{align}
	 (\textstyle{\frac 12}+\alpha+\beta+\gamma)(\alpha+\beta+\gamma)+
	(\textstyle{\frac 12}+\alpha)(\alpha)+
	(\textstyle{\frac 12}+\beta)(\beta)+
	(\textstyle{\frac 12}+\gamma)(\gamma)& \\
	 -(\textstyle{\frac 12}+\alpha+\beta)(\alpha+\beta)
	-(\textstyle{\frac 12}+\alpha+\gamma)(\alpha+\gamma)
	-(\textstyle{\frac 12}+\beta+\gamma)(\beta+\gamma)&=0.
	\end{align}
\end{proof}

\subsection{Special cases of the general formula}

In this section we use lemma~\ref{lem:generalformula} in order to derive the formulae for $c_0$, $c_1$, and $c_2$. Since all our expressions are homogeneous (the sum of the indices of the polynomials $Q$ is always equal to $g+1$), we can drop the factor $(-1)^m/2^m$ in the definition of $Q_m$, $m\geq 0$. 

We can substitute the values $Q_1(\frac 32) = 3$, $Q_1(1)=\frac 32$, $Q_1(\frac 12) = \frac 12$, $Q_1(0)=0$, $Q_2(\frac 52)=\frac{45}4$, $Q_2(2)=\frac{15}4$, $Q_2(\frac 32) = \frac 34$, $Q_3(\frac 52) = \frac{15}8$, $Q_3(2)=0$ in equation~\eqref{eq:generalformula}. This gives use the following coefficients of $Q_{g-n}(a_1)$, $Q_{g+1-n}(a_1)$, and $Q_{g+1-n}(a_1-\frac 12)$:
\begin{align} \label{eq:c0-terms}
\text{in }c_0:\qquad  &\left(\textstyle{\frac{3g}{2} -\frac 94 +\frac {3n}2}\right) Q_{g-n}(a_1) 
 + \left(\textstyle{6g +\frac 32 -3n}\right) Q_{g+1-n}(a_1)
\\ \notag
& + \left(\textstyle{-6g +0 +3n}\right) Q_{g+1-1}(a_n-\textstyle{\frac 12})
\\ \notag
\text{in }c_1:\qquad  &\left(\textstyle{\frac{3g}{2} -\frac {15}4 +\frac {3n}2}\right) Q_{g-n}(a_1) 
+ \left(\textstyle{6g +\frac 12 -3n}\right) Q_{g+1-n}(a_1)
\\ \notag
& + \left(\textstyle{-6g +1 +3n}\right) Q_{g+1-n}(a_1-\textstyle{\frac 12})
\\ \notag
\text{in }c_2:\qquad  &\left(\textstyle{\frac{3g}{2} -\frac {21}4 +\frac {3n}2}\right) Q_{g-n}(a_1) 
+ \left(\textstyle{6g -\frac 12 -3n}\right) Q_{g+1-n}(a_1)
\\ \notag
& + \left(\textstyle{-6g +2 +3n}\right) Q_{g+1-n}(a_1-\textstyle{\frac 12})
\end{align}
Note that the primary field $a_1$ has a different value in these three cases.

Furthermore, we are going to use that 
\begin{align}\label{eq:generalg1n-1}
& Q_{g+2-n}(a_1)-Q_{g+2-n}(a_1-\textstyle{\frac 12}) = \frac{(a_1)(a_1-\frac 12)\cdots (a_1-g-1+n)}{(g+1-n)!} \\
\label{eq:generalg1n-2}
& Q_{g+2-n}(a_1+1)-Q_{g+2-n}(a_1+\textstyle{\frac 32}) = \frac{-(a_1+\frac 32)(a_1+1)\cdots (a_1-g+\frac 12+n)}{(g+1-n)!} 
\end{align}

Let us combine these terms with the terms with $Q_{g+1-n}$ computed above. In the case of $c_0$ the primary field $a_1$ is equal to $2g-1 -\frac n2$. Then the sum of~\eqref{eq:c0-terms}, \eqref{eq:generalg1n-1}, and~\eqref{eq:generalg1n-2} is equal to the following expression:
\begin{align}
& \textstyle{\frac{(2g-1-\frac n2)\cdots (g-2+\frac n2)}{(g+1-n)!} }- (\textstyle{2g-1-\frac n2}) Q_{g+1-n} (\textstyle{2g-\frac 32 -\frac n2}) \\
& -(\textstyle{4g+1-\frac {5n}2}) Q_{g+1-n} (\textstyle{2g-\frac 32 -\frac n2}) +(\textstyle{4g+1-\frac {5n}2}) Q_{g+1-n} (\textstyle{2g-1 -\frac n2}) \\
& +(\textstyle{2g+\frac 12-\frac {n}2}) Q_{g+1-n} (\textstyle{2g-1 -\frac n2})-\frac{(2g+\frac 12 - \frac n2)\cdots (g-\frac 12+\frac n2)}{(g+1-n)!} \\
& =-(\textstyle{2g-1-\frac n2}) \frac{(2g-\frac 32 - \frac n2)\cdots (g-\frac 32+\frac n2)}{(g-n)!}
+(\textstyle{4g+1-\frac {5n}2}) \frac{(2g-1 - \frac n2)\cdots (g-1+\frac n2)}{(g-n)!} \\
& -(\textstyle{2g+\frac 12-\frac n2}) \frac{(2g-\frac 12 - \frac n2)\cdots (g-\frac 12+\frac n2)}{(g-n)!} \\
& =(\textstyle{3g+\frac 52-3n}) \frac{(2g-1 - \frac n2)\cdots (g-1+\frac n2)}{(g-n)!}
-(\textstyle{2g+\frac 12-\frac n2}) \frac{(2g-\frac 12 - \frac n2)\cdots (g-\frac 12+\frac n2)}{(g-n)!}
\end{align}
We can perform the same computation also for $c_1$ and $c_2$. Recall also in all three cases the term with $Q_{g-n}$ and the overall coefficients $\prod_{i=1}^{n-1} Q_1(a_i)$ in equation~\eqref{eq:generalformula}. We obtain the following expressions:

\begin{corollary} \label{cor:c-coeff} We have:
	\begin{align}
	c_0 =\ & Q_1(\textstyle{\frac 12})^{n-1} \left[ (\textstyle{\frac{3g}{2} -\frac 94 +\frac {3n}2}) \textstyle{\frac{(2g-\frac 12-\frac n2)\cdots (g-0+\frac n2)}{(g-n)!} } \right. \\ & \left.
	-(\textstyle{2g+\frac 12-\frac n2}) \frac{(2g-\frac 12 - \frac n2)\cdots (g-\frac 12+\frac n2)}{(g-n)!} + (\textstyle{3g+\frac 52-3n}) \frac{(2g-1 - \frac n2)\cdots (g-1+\frac n2)}{(g-n)!} \right] \\
	c_1 =\ & Q_1(\textstyle{\frac 12})^{n-2} Q_1(1) \left[ (\textstyle{\frac{3g}{2} -\frac {15}4 +\frac {3n}2}) \textstyle{\frac{(2g-1-\frac n2)\cdots (g-\frac 12+\frac n2)}{(g-n)!} } \right. \\ & \left.
	-(\textstyle{2g+0-\frac n2}) \frac{(2g-1 - \frac n2)\cdots (g-1+\frac n2)}{(g-n)!} + (\textstyle{3g+\frac 52-3n}) \frac{(2g-\frac 32 - \frac n2)\cdots (g-\frac 32+\frac n2)}{(g-n)!} \right] \\
		c_2 =\ & Q_1(\textstyle{\frac 12})^{n-3} Q_1(1)^2\left[ (\textstyle{\frac{3g}{2} -\frac {21}4 +\frac {3n}2}) \textstyle{\frac{(2g-\frac 32-\frac n2)\cdots (g-1+\frac n2)}{(g-n)!} } \right. \\ & \left.
		-(\textstyle{2g-\frac 12-\frac n2}) \frac{(2g-\frac 32 - \frac n2)\cdots (g-\frac 32+\frac n2)}{(g-n)!} + (\textstyle{3g+\frac 52-3n}) \frac{(2g-2 - \frac n2)\cdots (g-2+\frac n2)}{(g-n)!} \right] 
\end{align}
\end{corollary} 

\subsection{Proof of non-degeneracy}\label{subsec:proofnondegeneracy} In this subsection we prove proposition~\ref{prop:nondegeneracy}. First, observe that 
$Q_{g+1-n}(a_1)-Q_{g+1-n}(\textstyle{a_1-\frac12})$ is equal to $\textstyle{\frac{(a_1)(a_1-\frac 12)\cdots (a_1-g+n)}{(g-n)!}}$. We substitute $a_1=2g-1+\frac n2$ for $c_0$ (respectively, $2g-\frac 32+\frac n2$ for $c_1$ and $2g-2+\frac n2$ for $c_2$) and combine the result of corollary~\ref{cor:c-coeff} and equation~\ref{eq:chatdefinition} in order to obtain the following formulae:
\begin{align}
	\textstyle\frac 34 \hat c_0 =\ &  (\textstyle{\frac{3g}{2} -\frac 94 +\frac {3n}2}) \textstyle{\frac{(2g-\frac 12-\frac n2)}{(g-\frac 12+\frac n2)(g-1+\frac n2)}}
	-(\textstyle{2g+\frac 12-\frac n2}) \frac{(2g-\frac 12 - \frac n2)}{(g-1+\frac n2)} + (\textstyle{3g+\frac 52-3n}) \\
	\textstyle\frac 34  \hat c_1 =\ & (\textstyle{\frac{3g}{2} -\frac {15}4 +\frac {3n}2}) \textstyle{\frac{(2g-1-\frac n2)}{(g-1+\frac n2)(g-\frac 32+\frac n2)} } 
	-(\textstyle{2g+0-\frac n2}) \frac{(2g-1 - \frac n2)}{(g-\frac 32+\frac n2)} + (\textstyle{3g+\frac 52-3n}) \\
	\textstyle\frac 34 \hat c_2 =\ &  (\textstyle{\frac{3g}{2} -\frac {21}4 +\frac {3n}2}) \textstyle{\frac{(2g-\frac 32-\frac n2)}{(g-\frac 32+\frac n2)(g-2+\frac n2)} } 
	-(\textstyle{2g-\frac 12-\frac n2}) \frac{(2g-\frac 32 - \frac n2)}{(g-2+\frac n2)} + (\textstyle{3g+\frac 52-3n}) 
\end{align}
By an explicit computation, we obtain that 
\[
\textstyle\frac 34 (\hat c_0 -2\hat c_1 + \hat c_2) = \frac{S(g,n)}{(g-\frac 12+\frac n2)(g-1+\frac n2)(g-\frac 32+\frac n2)(g-2+\frac n2)},
\]
where 
\[
S(g,n)=-g+\frac{11}8 n-\frac 94 g^2+\frac 98 gn -\frac 12 g^3+\frac 34 g^2n-\frac 14 n^3
\]
We want to prove that this polynomial is never equal to zero in the integer points $(g,n)$ satisfying $3\leq n\leq g-1$. We can make a change of variable $n=b+3$, $g=a+b+4$, then we want to prove that $S(a+b+4, b+3)$ never vanishes for any integer $a,b\geq 0$. 
This is indeed the case since all non-zero coefficients of the polynomial
\[
S(a+b+4, b+3) =-\frac{201}8-\frac{173}8 a -\frac{21}2b -6a^2-\frac{39}8ab
-\frac 98b^2-\frac 12 a^3 - \frac 34 a^2b
\]
are negative including the constant term. This completes the proof of proposition~\ref{prop:nondegeneracy}.

\section{Vanishing of \texorpdfstring{\( R^{\geq g}(\mc{M}_{g,n})\)}{high tautological ring}}\label{sec:dimRg}

In this section we will give a new proof of the following theorem.
\begin{theorem}[\cite{Looijenga1995,Ionel2002}]\label{thm:highertautringzero}
The tautological ring of \( \mc{M}_{g,n}\) vanishes in degrees \( g\) and higher, that is \( R^{\geq g} (\mc{M}_{g,n}) = 0\).
\end{theorem} 
This theorem and theorem~\ref{thm:BSZ16} together consistute the generalized socle conjecture, as the bound \( \dim R^{g-1} (\MMM_{g,n} ) \geq n\) can be proved relatively simply, see e.g. \cite{BuryakShadrinZvonkine2016}. This conjecture is a generalization of one of Faber's three conjectures on the tautological ring of \( \MMM_g \), see \cite{Faber1999} for the original conjectures and \cite{BuryakShadrinZvonkine2016} for the generalization.\par
The proof consists of three steps: in steps one and two, we show that the pure \( \psi \)- and \( \kappa \)-classes vanish, respectively, and in step three we reduce the mixed monomials to the pure cases. The first two steps will be proved in separate lemmata.

%\subsection{Vanishing of monomials of \texorpdfstring{\( \psi \)}{psi}-classes}\label{subsec:nopsiing}

\begin{lemma}\label{lem:nopsiing}
Let \( g \geq 0\) and \( n \geq 1\). Any monomial in \( \psi \)-classes of degree at least \( \max (g, 1) \) vanishes on \( \mc{M}_{g,n} \).
\end{lemma}
\begin{remark}
This lemma was originally conjectured by Getzler in \cite{Getz98}.
\end{remark}
\begin{proof} 
For \( g= 0\), this is well-known, see e.g. \cite[proposition 2.13]{Zvonkine2012}. So let us assume \( g \geq 1\).\par
We will prove that any monomial in \( \psi \)-classes of degree \( g\) vanishes. This clearly implies that any monomial of higher degree vanishes as well.\par
For this, look again at \( \Omega^g \), but now on \( \overline{\mc{M}}_{g,n} \). When restricted to the open part \( \mc{M}_{g,n} \), the only contributing graph is the one with one vertex of genus \( g\), as the other graphs correspond to boundary divisors by definition. Hence, the equation for the CohFT reduces to
\begin{equation}
\Omega_{g,n}^{g}(a_1, \dotsc, a_n) \bigg|_{\mc{M}_{g,n}}= \begin{cases} -\frac{1}{2}  \prod_{i=1}^n \Big( \sum_{m_i \geq 0} Q_{m_i}(a_i) \psi_i^{m_i} \Big) &\text{if } \sum_{i=1}^n a_i = 2g-1 \\ 0 & \text{else.} \end{cases}
\end{equation}
We will prove vanishing of all monomials using downward induction on the exponent \( d_1 \) of \( \psi_1 \), starting with the case of \( d_1 = g+1\). This case trivially gives a zero, as this power cannot occur in a monomial of total degree \( g\).\par
Now, assuming all monomials with exponent of \( \psi_1 \) larger than \( d_1 \) vanish, consider the monomial \( \psi_1^{d_1} \dotsb \psi_n^{d_n} \) for any \( d_i \) summing up to \( g\). For the relation, choose \( a_i = d_i\) for all \( i \neq 1\), and \( a_1 = 2g-1 - \sum_{i=2}^n a_i \). This means \( Q_{m_i} (a_i) = 0\) unless \( m_i \leq d_i\) or \( i = 1\), so the only monomials with non-zero coefficients have exponent of \( \psi_i \) at most \( d_i \) for \( i \neq 1\). Because the total degree is fixed, the only surviving monomial with exponent of \( \psi_1 \) equal to \( d_1 \) is the one we started with, and this relation expresses it in monomials with strictly larger exponent of \( \psi_1 \). By the induction hypothesis, this monomial must be zero.
\end{proof}

\begin{remark}
Note that this argument breaks down for degrees lower than \( g\), as the class does not vanish there. Therefore, to get relations in those degrees, one must push forward relations in higher degrees along forgetful maps \emph{on the compactified moduli space}, which contain non-trivial contributions from boundary strata.
\end{remark}

%\subsection{Vanishing of pure \texorpdfstring{\( \kappa \)}{k}-classes}\label{subsec:nokappainhighd}

\begin{lemma}\label{lem:nokappainhighd}
Any multi-index \( \kappa \)-class of degree at least \( g\) vanishes on \( \mc{M}_{g,n} \).
\end{lemma}
\begin{proof}
Fix a degree \( d \geq g\), and consider the pure (multi-index) \( \kappa \)-classes in this degree. Without loss of generality, we can assume the amount of indices to be equal to \( d\): this is certainly an upper bound, and adding and extra zero index only multiplies the class by a non-zero factor, using the dilaton equation on the definition of multi-index \( \kappa \)-classes.\par
We will consider \( \Omega^g_{g,n+d} \). In order to get a relation in \( R^d (\mc{M}_{g,n}) \), we should multiply by a class \( \sigma \) of degree \( 2d-g\), push forward to \( \overline{\mc{M}}_{g,n} \), and then restrict to \( \mc{M}_{g,n} \). As we can now assume \( d \geq g\), we have \( 2d-g \geq d\), and we can therefore choose \( \sigma = \prod_{j=1}^d \psi_{n+j}^{f_j +1} \), with each \( f_j \geq 0\). By choosing such a \( \sigma \), we ensure that after pushforward and restriction to the open moduli space, none of the contributions from boundary divisors on \( \oM_{g,n+d} \) survive, and only the term with one vertex contributes.\par
We will use downward induction on the first index of the \( \kappa \)-class. The base case is a first index larger than \( d\), and hence another index being negative, giving a trivial zero.\par
Now, assume all \( \kappa \)-classes with first index larger than \( e_1 \) are zero. Fix a class \( \kappa_{e_1, \dotsc, e_d} \) of degree \(d = \sum_{j=1}^d e_j\), choose a set of non-negative integers \( \{ a_j, f_j \mid 2 \leq j \leq d \} \) such that \( a_j + f_j = e_j\), and set \( a_1 = 2g-1- \sum_{j=2}^d a_j\) and \( f_1 = 0\). We will consider
\begin{align}
&\pi_*^d \big( \sigma \cdot \Omega_{g,n+d}^g(0, \dotsc, 0, a_1, \dotsc, a_d) \big) \bigg|_{\mc{M}_{g,n}}\\
&\qquad=  \pi_*^d \Big( \prod_{j=1}^d \psi_{n+j}^{f_j+1} \cdot -\frac{1}{2} \prod_{j=1}^d \sum_{m_j\geq 0} Q_{m_j} (a_j) \psi_{n+j}^{m_j} \Big) \bigg|_{\mc{M}_{g,n}} \\
&\qquad= -\frac{1}{2} \sum_{\substack{m_j \geq 0\\1 \leq j \leq d}} \Big( \prod_{j=1}^d Q_{m_j}(a_j) \Big) \kappa_{f_1+m_1, \dotsc, f_d+m_d},
\end{align}
which vanishes. By our choice of \( a_j \), for the product of \( Q\)-polynomials to be non-zero, we need \( m_j \leq a_j \) for \( j \neq 1\). Furthermore, by our choice of \( f_j\), this shows that \( f_j + m_j \leq e_j \) for \( j \neq 1\). Because we look at a fixed degree \( d\), this means \( f_1 + m_1 \geq e_1 \), with equality only occuring for \( m_j = f_j \), \( j \neq 1\), and hence for the \( \kappa \)-class we started with. Hence this relation expresses our chosen class \( \kappa_{e_1, \dotsc, e_d} \) in terms of \( \kappa \)-classes with strictly higher first index, which we already know vanish.
\end{proof}
\begin{remark}
Note that we cannot use the vanishing of the \( \psi \)-monomials in higher degrees and push these relations forward, as the \( \kappa \)-classes are defined by pushing forward \( \psi \)-classes on the compactified moduli space and then restricting to the open part, and not the other way around.
\end{remark}

We are now ready to prove the theorem.

\begin{proof}[Proof of theorem~\ref{thm:highertautringzero}]
For general monomial \( \psi \)-\( \kappa \)-classes, i.e. classes of the form \( \mu = \psi_i^{d_1} \dotsb \psi_n^{d_n} \cdot \kappa_{e_1, \dotsc, e_k} \), we will use induction on the total degree \( d = \sum_{i=1}^n d_i + \sum_{j=1}^k e_j \). If all \( d_i \) are zero, we are in the case of lemma~\ref{lem:nokappainhighd}, so we can assume at least one of them is non-zero, i.e. \( \mu = \nu \cdot \psi_i \) for some \( i\).\par
In degree \( d= g\), we get that the degree of \( \nu \) is \( g-1\). By proposition~\ref{prop:psispanRg-1}, we know that \( \nu \) is a polynomial in \( \psi \)-classes. Therefore, so is \( \mu = \nu \cdot \psi_i \). By lemma~\ref{lem:nopsiing}, we know \( \mu \) vanishes.\par
For the induction step, we know by induction that \( \nu \) is zero, hence \( \mu \) is too. This finishes the proof of theorem~\ref{thm:highertautringzero}.
\end{proof}

Because the proof of this theorem only uses the case \( x=0\) from subsection~\ref{subsection:analysisrelations}, see also subsection~\ref{subsec:simprelI}, and only fixed non-negative integer primary fields, all the relations are actually explicit on all of \( \oM_{g,n} \). Hence, we get the following
\begin{proposition}\label{prop:classalgbound}
The Pandharipande-Pixton-Zvonkine relations for \( r = \frac12 \) give an algorithm for computing explicit tautological boundary formulae in the Chow ring for any tautological class on \( \oM_{g,n} \) of codimension at least \( g\). In particular, the intersection numbers of \( \psi \)-classes on \( \oM_{g,n} \) can be computed with these relations for any \( g \geq 0\) and \( n \geq 1\) such that \( 2g-2+n >0\).
\end{proposition}
\begin{remark}
The first part of the statement is very similar to \cite[Theorem 5]{CladerGrushevskyJandaZakharov2016}, which gave a reduction algorithm based on Pixton's double ramification cycle. It confirms an expectation on \cite[Page 7]{BJP15}, that ``(...)Pixton's relations are expected to uniquely determine the descendent theory, but the implication is not yet proven.''
\end{remark}
Note that the intersection numbers in \( \psi \)- and \( \kappa \)-classes can be expressed as intersection numbers of only \( \psi \)-classes by pulling back along forgetful maps, see \cite[Corollary 3.23]{Zvonkine2012}. By the proposition, all these intersection numbers can then be computed using the PPZ relations.
\begin{proof}
The first sentence follows by the comment above the proposition. For the second sentence, we will reduce polynomials in \( \psi \)-classes to smaller and smaller boundary strata using our explicit relation. This will be done in the form of an induction on \( \dim \oM_{g,n} = 3g-3+n\), the zero-dimensional case \( \oM_{0,3} \) being obvious.\par
For any \( g_1 + g_2 = g\) and \( I_1 \sqcup I_2 = \{ 1, \dotsc, n\} \) such that \( 2g_i + |I_i| -1 > 0\), write \( \rho_{I_1,I_2}^{g_1,g_2} \colon \oM_{g_1,|I_1|+1} \times \oM_{g_2,|I_2|+1}  \to \oM_{g,n} \) for the attaching map, and \( D_{I_1,I_2}^{g_1,g_2} \) for the divisor \( (\rho_{I_1,I_2}^{g_1,g_2})_*(1) \). Similarly, write \( \sigma : \oM_{g-1,n+2} \to \oM_{g,n} \) for the glueing map, and \( \delta_\textrm{irr} \) for \( \sigma_* (1)\). Then these divisors together form the entire boundary of \( \oM_{g,n} \), and \( \rho^* (\psi_i ) = \psi_i \) and \( \sigma^* (\psi_i ) = \psi_i \) for any choice of indices.\par
Now let \( g\) and \( n\) be such that \( 3g-3+n > 0\), and choose a polynomial \( p (\psi ) \in R^{3g-3+n}(\oM_{g,n})\). Using stability, \( 3g-3+n > g-1\), so by lemma~\ref{lem:nopsiing}, this class is zero on \( \MMM_{g,n} \). Since the proof only uses relations without \( \kappa \)-classes, it can be given explicitly as a sum of the boundary divisors given above multiplied with other \( \psi \)-polynomials. By the projection formula,
\begin{align}
\int_{\oM_{g,n}} \prod_{i=1}^n \psi_i^{d_i} D_{I_1,I_2}^{g_1,g_2} (\psi')^{d'} (\psi'')^{d''} &= \int_{\oM_{g_1,|I_1|+1}} \!\!\! \Big( \prod_{i \in I_1} \psi_i^{d_i}\Big) \psi_{n+1}^{d'} \cdot \int_{\oM_{g_2,|I_2|+1}} \!\!\! \Big( \prod_{i \in I_2} \psi_i^{d_i} \Big) \psi_{n+2}^{d''};\\
\int_{\oM_{g,n}} \prod_{i=1}^n \psi_i^{d_i} \delta_\textrm{irr} (\psi')^{d'} (\psi'')^{d''} &= \int_{\oM_{g-1,n+2}} \!\!\! \Big( \prod_{i =1}^n \psi_i^{d_i}\Big) \psi_{n+1}^{d'} \psi_{n+2}^{d''},
\end{align}
where \( \psi'\) and \( \psi'' \) are the classes on the half-edges of the unique edge in the dual graphs of the divisors.\par
All spaces on the right-hand side have a strictly lower dimension, so by induction we can compute those numbers via the PPZ relations.
\end{proof}

According to \cite[Subsection 3.5]{CladerGrushevskyJandaZakharov2016}, proposition~\ref{prop:classalgbound} implies the following theorem.

\begin{corollary}[Theorem \( \star \) \cite{GrVa05}, improved in \cite{FaPa05}]\label{cor:thmstar}
Any codimension \( d\) tautological class can be expressed in terms of tautological classes supported on curves with at least \( d - g +1 \) rational components.
\end{corollary}

\section{Dimensional bound for \texorpdfstring{\( R^{\leq g-2}(\mc{M}_{g,n})\)}{low tautological ring}}\label{sec:bounddimRlow}

Similarly to \cite[theorem 6]{PPZ16}, our method also gives a bound for the dimension of the lower degree tautological classes. For the statement of this proposition, recall that \( p(n) \) denotes the number of partitions of \( n\), and \( p(n,k)\) denotes the number of partitions of \( n \) of length at most \( k\).
\begin{proposition}
\begin{equation}
\dim R^d (\mc{M}_{g,n}) \leq \sum_{k=0}^d  \binom{n+k-1}{k} p(d-k,g-1-d)
\end{equation}
\end{proposition}
\begin{remark}
If we use the natural interpretation of \( \binom{k-1}{k} \) as \( \delta_{k,0} \), this does indeed recover \cite[theorem 6]{PPZ16} in the case \( n =0 \).
\end{remark}
\begin{proof}
We will exhibit an explicit spanning set of this cardinality, consisting of \( \psi \)-\( \kappa \)-classes: monomials in \( \psi \)-classes multiplied with a multi-index \( \kappa \)-class.\par
First, a less strict first bound can be obtained as follows: any \( \psi \)-\( \kappa \)-class has a definite degree in \( \psi \)'s, say \( k\). There are \( \binom{n+k-1}{k} \) different monomials of degree \( k\) in \( n\) variables, and furthermore there are as many different multi-index \( \kappa \)-classes of degree \( d-k \) as there are partitions of \( d-k \), so \( p(d-k) \). This gives the first bound
\begin{equation}
\dim R^d (\mc{M}_{g,n}) \leq \sum_{k=0}^d  \binom{n+k-1}{k} p(d-k),
\end{equation}
which is already close to the statement of the proposition.\par
To get the actual bound, we will show that any \( \psi \)-\( \kappa \)-class with at least \( g-d \) \( \kappa \)-indices can be expressed in \( \psi \)-\( \kappa \)-classes with strictly fewer \( \kappa \)-indices. Following the logic of the previous paragraph, this proves the bound.\par
This reduction step is analogous to the proof of lemma~\ref{lem:nokappainhighd}. Suppose we have a class \( \mu = \psi_1^{d_1} \dotsb \psi_n^{d_n} \kappa_{e_1, \dotsb e_m} \) with \( m \geq g-d\). Choose non-negative integers \( \{ f_i, a_i\}_{i=1}^{n+m} \) such that the following hold:
\begin{align}
f_1 &= 0;&&\\
  \sum_{i=1}^{n+m} f_i &= d-g+m; &&\\
a_i + f_i &= d_i, &\text{for } &2 \leq i \leq n;\\
a_{n+j} + f_{n+j} &= e_j +1, &\text{for } &1 \leq j \leq m;\\
 a_1 &= 2g-1 - \sum_{j=2}^m a_j.&&
\end{align}
Let \( \sigma = \prod_{i=2}^{n+m} \psi_{i}^{f_i} \), and consider the class
\begin{equation}
\pi_*^m \big( \sigma \cdot \Omega_{g,n+m}^g(a_1, \dotsc, a_{n+m}) \big) \bigg|_{\mc{M}_{g,n}}.
\end{equation}
By the second condition on our chosen numbers, which fixes the degree of \( \sigma \), this expression gives a relation in \( R^d (\mc{M}_{g,n})\).\par
There are no \( \psi \)-\( \kappa \)-classes with more than \( m\) \( \kappa \)-indices in this relation, and the coefficient of any \( \psi\)-\( \kappa \)-class with exactly \( m\) indices can only come from the open part of \( \overline{\mc{M}}_{g,n+m} \), as each forgotten point must carry at least two \( \psi \)-classes, which would give too high degrees on any rational component. Therefore, the coefficient of \( \psi_1^{p_1} \dotsb \psi_n^{p_n} \kappa_{q_1, \dotsc, q_m} \) must be \( \prod_{i=1}^n Q_{p_i-f_i}(a_i) \cdot \prod_{j=1}^m Q_{q_j-f_{n+j}+1}(a_{n+j}) \). This is only non-zero if \( p_i \leq f_i + a_i = d_i \) for all \( i \neq 1 \) and \( q_j \leq f_{n+j}+ a_n+j-1 = e_j \) for all \( j \). This implies that \( p_1 \geq d_1 \), with equality only if \( p_i = d_i \) and \( q_j = e_j \) for all \( i, j\). Hence, this relation expresses the class \( \mu \) as a linear combination of \( \psi \)-\( \kappa \)-classes with less than \( m\) \( \kappa \)-indices and \( \psi \)-\( \kappa \)-classes with strictly higher exponent of \( \psi_1 \). By induction on first the exponent of \( \psi_1 \) and then the number of \( \kappa \)-indices, all these classes can be reduced.
\end{proof}
\begin{remark}
This argument breaks down for \( m < g-d\), as the class \( \sigma \) would have to have a negative degree: our class only vanishes in degree at least \( g\), and to get at most \( m\)-index \( \kappa \)-classes, we can only push forward \( m\) times, so the lowest degree relation would be in \( R^{g-m} \).
\end{remark}
The condition that partitions have length at most \( g - 1 - d\) seems dual to Graber and Vakil's Theorem \( \star \), corollary~\ref{cor:thmstar}, see \cite[theorem 1.1]{GrVa05}.%, which states that any tautological class of degree \( d\) on \( \oM_{g,n} \) vanishes on the open set consisting of strata with less than \( d- g+1\) components of genus \( 0\).

\bibliographystyle{alpha}
\bibliography{tautologicalrelationsonMgn}

\end{document}